\documentclass[journal]{IEEEtran}

\usepackage{etex}
\newcommand{\subparagraph}{}
\usepackage{hyperref}
\usepackage{amsmath}
\usepackage{amsthm}
\usepackage{bbm}
\usepackage[utf8]{inputenc}
\usepackage{graphicx}
\usepackage{caption}
\usepackage{subcaption}
\usepackage{framed, color} 
\usepackage[margin=.7in]{geometry}
\usepackage{amssymb}
\usepackage{amsfonts}
\usepackage{cite}
\usepackage{ifthen}
\usepackage{epsfig}
\usepackage{subcaption}
\usepackage{array}
\newcolumntype{C}[1]{>{\centering\let\newline\\\arraybackslash\hspace{0pt}}m{#1}}
\usepackage{enumerate}
\usepackage{algorithm}
\usepackage[all]{xy}
\usepackage{algpseudocode}
\usepackage[compact]{titlesec}
\usepackage{footmisc}
\usepackage{times}
\usepackage{tikz}
\usepackage{float}
\usetikzlibrary{shapes,positioning,arrows,automata}
\usepackage{caption}
\usepackage{nth}
\usepackage{array} 
\theoremstyle{plain}
\newtheorem{theorem}{Theorem}

\newtheorem{lemma}{Lemma}
\newtheorem{proposition}{Proposition}
\newtheorem{corollary}{Corollary}

\theoremstyle{definition}
\newtheorem{definition}{Definition}
\newtheorem{example}{Example}
\theoremstyle{remark}
\newtheorem{remark}{Remark}
\usepackage{array}

\newtheorem{claim}{Claim}
\theoremstyle{definition}
\DeclareMathOperator*{\argmin}{arg\,min}

\newcommand{\phsc}{\Phi_{\text{HSC}}}
\newcommand{\phsclr}{\Phi_{\text{HSCLR}}}
\newcommand{\era}{{\sf x}}
\newcommand{\nn}{\nonumber}
\newcommand{\ind}{\mathbb{I}}
\newcommand{\bern}{{\sf Bern}}
\newcommand{\caa}{c_0}
\newcommand{\caaa}{c_1}

\newcommand{\cccc}{c_3}
\newcommand{\cddd}{c_4}
\newcommand{\ceee}{c_5}
\newcommand{\cfff}{c_6}
\newcommand{\cseven}{c_7}
\newcommand{\ceight}{c_8}
\newcommand{\cnine}{c_{9}}
\newcommand{\cten}{c_{10}}
\newcommand{\celeven}{c_{11}}
\newcommand{\ctwelve}{c_{12}}
\newcommand{\cthirteen}{c_{13}}
\newcommand{\cfourteen}{c_{14}}
\newcommand{\cfifteen}{c_{15}}

\newcommand{\cost}{\text{cost}}
\newcommand{\fit}{\text{fit}}

\newcommand{\ww}{\mathbf{W}}

\newcommand{\ee}{\eta}
\newcommand{\Az}{\mathbf{A}^0}

\newcommand{\nei}{\mathcal{E}^{(i)}}

\newcommand{\cweakone}{c_{0}}
\newcommand{\cweaktwo}{c_{1}}
\newcommand{\cdetec}{c_{2}}

\newcommand{\ctt}{c_{\text{thr}}}
\usepackage{xr}
\newcommand{\den}{\text{den}}

\newcommand{\Uz}{\mathbf{U}^0}

\newcommand{\dd}{{\sf d}}

\newcommand{\eeq}{\end{equation}}
\newcommand{\bea}{\begin{eqnarray}}
	\newcommand{\eea}{\end{eqnarray}}
\newcommand{\bean}{\begin{eqnarray*}}
	\newcommand{\eean}{\end{eqnarray*}}
\newcommand{\bit}{\begin{itemize}}
	\newcommand{\eit}{\end{itemize}}
\newcommand{\ben}{\begin{enumerate}}
	\newcommand{\een}{\end{enumerate}}
\newcommand{\blem}{\begin{lem}}
	\newcommand{\elem}{\end{lem}}
\newcommand{\bthm}{\begin{thm}}
	\newcommand{\ethm}{\end{thm}}
\newcommand{\bpf}{\begin{IEEEproof}}
	\newcommand{\epf}{\end{IEEEproof}}

\newcommand{\comment}[1]{}

\newcommand*\err{err}

\usepackage{algorithm}\usepackage{algpseudocode}
\usepackage{mathtools}
\usepackage{soul}

\newcommand{\ye}{f_e}
\newcommand{\Ye}{\mathbf{F}}
\newcommand{\var}{\textbf{Var}}

\newcommand{\ex}{\mathbb{E}}

\newcommand\naive{na{\"i}ve}
\newcommand\rev[2]{#2}

\usepackage{epstopdf}

\newcommand*\errr{\text{err}'}
\newcommand{\mgf}{\mathbb{M}}

\newcommand{\cmatrix}{\log \left(\frac{1-\theta}{\theta}\right)}

\begin{document}

\title{}

\titleformat{\subsubsection}[runin]{\normalfont}{\arabic{subsubsection})}{0.5ex}{}[:]

\title{Hypergraph Spectral Clustering\\
in the Weighted Stochastic Block Model}
\author{
	Kwangjun Ahn, Kangwook Lee, and Changho Suh
	\thanks{
		
		This work was supported by the National Research Foundation of Korea (NRF) grant funded by the Korea government (MSIT) (No. 2015R1C1A1A02036561).
			\rev{}{This paper was presented in part at the IEEE International Symposium on Information Theory 2017~\cite{Ahn1706:Information}.}
			
		Kwangjun Ahn is with the Department of Mathematical Sciences, KAIST (e-mail: kjahnkorea@kaist.ac.kr).		 
		
		Kangwook Lee and Changho Suh are with the School of Electrical Engineering, KAIST (e-mail: $\{$kw1jjang, chsuh$\}$@kaist.ac.kr).		
	
	}}

\date{}
\maketitle

\begin{abstract} 
	
	Spectral clustering is a celebrated algorithm that partitions objects based on pairwise similarity information. While this approach has been successfully applied to a variety of domains, it comes with limitations. The reason is that there are many other applications in which only \emph{multi}-way similarity measures are available. This motivates us to explore the multi-way measurement setting. In this work, we develop two algorithms intended for such setting:  Hypergraph Spectral Clustering (HSC) and Hypergraph Spectral Clustering with Local Refinement (HSCLR). 
	Our main contribution lies in performance analysis of the poly-time algorithms under a random hypergraph model, which we name the weighted stochastic block model, in which objects and multi-way measures are modeled as nodes and weights of hyperedges, respectively.
	Denoting by $n$ the number of nodes, our analysis reveals the following: (1) HSC outputs a partition which is better than a random guess if the sum of edge weights (to be explained later) is $\Omega(n)$; (2) HSC outputs a partition which coincides with the hidden partition except for a vanishing fraction of nodes if the sum of edge weights is $\omega(n)$; 
	and (3) HSCLR exactly recovers the hidden partition if the sum of edge weights is on the order of $n \log n$.
	Our results improve upon the state of the arts recently established under the model and  they firstly settle the order-wise optimal results \rev{}{for the binary edge weight case}. 
	Moreover, we show that our results lead to efficient sketching algorithms for subspace clustering, a computer vision application.  
	Lastly, we show that HSCLR achieves the information-theoretic limits for a special yet practically relevant model, thereby showing no computational barrier for the case.
	
\end{abstract} 

\section{Introduction}\label{sec:intro}

\IEEEPARstart{T}{he} problem of clustering is prevalent in a variety of applications such as social network analysis, computer vision, and computational biology.
Among many clustering algorithms, \emph{spectral clustering} is one of the most prominent algorithms proposed by \cite{shi2000normalized} in the context of image segmentation, viewing an image as a graph of pixel nodes, connected by weighted edges representing visual similarities between two adjacent pixel nodes.
This approach has become popular, showing its wide applicability in numerous applications, and has been extensively analyzed under various models~\cite{mcsherry2001spectral, rohe2011spectral, lei2015consistency} . 

While  the standard spectral clustering relies upon interactions between \emph{pairs of two nodes}, there are many applications where interaction occurs across more than two nodes.
One such application includes a social network with online social communities, called folksonomies, in which users attach tags to resources. 
In the example, a three-way interaction occurs across users, resources and annotations~\cite{ghoshal2009random}.
Another application is molecular biology, in which multi-way interactions between distinct systems capture molecular interactions~\cite{michoel2012alignment}. 
See \cite{JMLR:v18:16-100} and the list of applications therein. 
Hence, one natural follow-up research direction is to extend the celebrated framework of \emph{graph} spectral clustering into a \emph{hypergraph} setting in which edges reflect multi-way interactions.

As an effort, in this work, we consider a random weighted uniform hypergraph model which  we call \emph{the weighted stochastic block model}, which is a special case of that considered in~\cite{JMLR:v18:16-100}. 
 An edge of size $d$ is \emph{homogeneous} if it consists of nodes from the same group, and is \emph{heterogeneous} otherwise.\footnote{\rev{}{While edges of a graph are pairs of nodes, edges of a hypergraph (or hyperedges) are arbitrary sets of nodes. Further, the size of an edge is the number of nodes contained in the edge.}}
Given a hidden partition of $n$ nodes into $k$ groups, a weight is   independently  assigned to each edge of size $d$ such that homogeneous edges tend to have higher weights than heterogeneous edges.  
More precisely, for some constants $p > q$,  the expectation of homogeneous edges' weights is $p\alpha_n$ and that of heterogeneous edges' weights is $q\alpha_n$.\footnote{For illustrative purpose, we focus on a symmetric setting. In Sec.~\ref{sec:dis}, we will extend our results (to be described later) to a more general setting.}
Here,  $\alpha_n$ captures the sparsity level of the weights, which may decay in $n$.
The task here is to recover the hidden partition from the weighted hypergraph. 
In particular, we aim to develop computationally efficient algorithms that provably find the hidden partition.

{\bf Our contributions:} 
By generalizing the spectral clustering algorithms proposed for the graph clustering, we first propose two   poly-time  algorithms  which  we name \emph{Hypergraph Spectral Clustering} (HSC) and \emph{Hypergraph Spectral Clustering with Local Refinement} (HSCLR).
 We then  analyze  their performances, 
 assuming that the size of hyperedges is $d$, the number of clusters $k$ is constant, and the size of each group is linear in $n$. 
 Our main results can be summarized as follows.
 For some constants $c$ and $c'$, which depend only on $p$, $q$, and $k$, the following statements hold with high probability: 
\begin{itemize}
	\item \emph{Detection:}  If $\binom{n}{d}\alpha_n \geq c\cdot n$, the output of HSC is more consistent with the hidden partition than a random guess; 
	\item \emph{Weak consistency:}  If $\binom{n}{d}\alpha_n =\omega(n)$, HSC outputs a partition which coincides with the hidden partition except $o(n)$ number of nodes; and
	\item \emph{Strong consistency:}  If $\binom{n}{d}\alpha_n\geq c'\cdot n\log n$, HSCLR exactly recovers the hidden partition.
\end{itemize}
We remark that our main results are the first order-wise optimal results \rev{}{for the binary edge weight case (see Proposition~\ref{optimal})}. 

 \begin{table}[t]
	\caption{Comparison to the state of the arts: ``Weak" and ``Strong'' are for consistency results.}
	\label{table1}
	\begin{center}
		\begin{scriptsize}
			\begin{tabular}{c||cc|ccc}
				\hline
				& \multicolumn{2}{c}{Model assumption} & \multicolumn{3}{c}{ Order of $\binom{n}{d}\alpha_n$ required for}   \\\hline  
				& \begin{tabular}{@{}c@{}}Multi\\Groups\end{tabular}  &\begin{tabular}{@{}c@{}}Weighted\\Edges\end{tabular} & Detection & Weak & Strong \\
				\hline
				\cite{ghoshdastidar2014consistency}    & $\surd$ & $\times$ & NA & $n^d$ &NA \\
				\cite{ghoshdastidar2015provable} & $\surd$& $\times$ & NA & $n^d$ &NA\\  
				\cite{ghoshdastidar2015consistency} & $\surd$& $\times$ & NA & $\Omega(n(\log n)^2)$ &NA\\ 
				\cite{JMLR:v18:16-100} & $\surd$& $\surd$ & NA & $\Omega(n(\log n)^2)$ & NA\\  
				\cite{florescu2015spectral}    & $\times$  &     $\times$ & $\Omega(n)$ & NA& NA \\ \hline
				\textbf{Ours}     &$\surd$&   $\surd$ & $\Omega(n)$ & $\omega(n)$ & $\Omega(n\log n)$\\
				\hline
			\end{tabular}
		\end{scriptsize}
	\end{center}
\vskip -0.2in
\end{table}

\subsection{Related work}\label{relatedwork}


\subsubsection{Graph Clustering}
The problem of standard graph clustering, i.e., $d=2$, has been studied in great generality.
Here, we summarize some major developments, referring the readers to a recent survey by Abbe~\cite{abbe2017community} for details.
The detection problem, whose goal is to find a partition that is more consistent with the hidden partition than a random guess, has received a wide attention.
A notable work by Decelle et al.~\cite{decelle} firstly observes phase transition and conjectures the transition limit. 
Further, they also conjecture that the computational gap exists for the case of $k \geq 4$. 
For the case of $k=2$, the phase transition limit is fully settled jointly by~\cite{mossel2015reconstruction} and \cite{Mas14, bordenave}: 
The impossibility of the detection below the conjectured threshold is established in~\cite{mossel2015reconstruction}, and it is proved that the conjectured threshold can be achieved via some efficient algorithms in~\cite{ Mas14,bordenave}.
The limits for the case $k \geq 3$ have been studied in~\cite{YC14, NN14, Mon15,BM16}, and are settled in~\cite{abbe2015detection}.

The weak/strong consistency problem aims at finding a cluster that is correct except a vanishing or zero fraction. 
The necessary and sufficient conditions for weak consistency have been studied in~\cite{AL14,gao2015achieving,MNS14a,yun2014accurate,abbe2015community}, and those for strong consistency in~\cite{abbe2016exact,7523889,MNS14a,abbe2015community}.
In particular for strong consistency, both the fundamental limits and computationally efficient algorithms are investigated initially for $k=2$~\cite{abbe2016exact,7523889,MNS14a}, and recently for general $k$~\cite{abbe2015community}.
While most of the works assume that the graph parameters such as $p$, $q$, $k$, and the size of clusters are fixed, one can also study the minimax scenario where the graph parameters are adversarially chosen against the clustering algorithm. 
In \cite{zhang2016minimax}, the authors characterize the minimax-optimal rate.
Further, \cite{gao2015achieving} shows that the minimax-optimal rate can be achieved by an efficient algorithm.

\subsubsection{Hypergraph Clustering}
Compared to graph clustering, the study of hypergraph clustering is still in its infancy.
In this section, we briefly summarize recent developments.
For detection, analogous to the work by Decelle et al.~\cite{decelle}, Angelini et al.~\cite{angelini} firstly conjecture phase transition thresholds.
These conjectures have not been settled yet unlike the graph case. 
In~\cite{JMLR:v18:16-100}, the authors study a specific spectral clustering algorithm, which can be shown to detect the hidden cluster if $\binom{n}{d}\alpha_n = \Omega(n (\log n)^2)$, while the conjectured threshold for detection is $\binom{n}{d}\alpha_n = c^\star n$ for some constant $c^\star$. 
Actually, this gap is due to the technical challenge that is specific to the hypergraph clustering problem: See Remark~\ref{rmk2} for details.
In~\cite{florescu2015spectral}, the authors study the bipartite stochastic block model, and as a byproduct of their results, they show that detection is possible under some specific model if $\binom{n}{d}\alpha_n = \Omega(n)$.
While this guarantee is order-wise optimal, it holds only when edge weights are binary-valued and the size of two clusters are equal. 
Our detection guarantee, obtained by delicately resolving the technical challenges specific to hypergraphs, is also order-wise optimal but does not require such assumptions.

While several consistency results under various models are shown in~\cite{ghoshdastidar2014consistency,ghoshdastidar2015provable, ghoshdastidar2015consistency, JMLR:v18:16-100,florescu2015spectral}, to the best of our knowledge, our consistency guarantees are the first order-wise optimal ones.
We briefly overview the existing results below.
In~\cite{ghoshdastidar2014consistency, ghoshdastidar2015provable}, the authors derive consistency results for the case in which $\alpha_n=1$ and weights are binary-valued.
In~\cite{JMLR:v18:16-100}, the authors investigate consistency results of a certain spectral clustering algorithm under a fairly general
random hypergraph model, called the planted partition model in hypergraphs.
Indeed, our hypergraph model is a special case of the planted partition model, and hence the algorithm proposed in~\cite{JMLR:v18:16-100} can be applied to our model as well. 
One can show that their algorithm is weakly consistent if $\binom{n}{d}\alpha_n = \Omega(n(\log n)^{2})$ under our model.
The case of non-uniform hypergraphs, in which the size of edges may vary, is studied in~\cite{ghoshdastidar2015consistency}.
See Table~\ref{table1} for a summary.

While most of the existing works focus on analyzing the performance of certain clustering algorithms, some study the fundamental limits.
In~\cite{ahn2016community,Ahn1706:Information}, the information-theoretic limits are characterized for specific hypergraph models.
In~\cite{wangetal}, the minimax optimal rates of error fraction are derived for the binary weighted edge case. 
However, it has not been clear whether or not a computationally efficient algorithm can achieve such limits. 
In this work, we show that HSCLR achieves the fundamental limit for the model considered in~\cite{Ahn1706:Information}.

\subsubsection{Main innovation relative to \cite{Ahn1706:Information}}
\rev{}{
The new algorithms proposed in this work can be viewed as strict improvements over the algorithm proposed in our previous work~\cite{Ahn1706:Information}. First, the algorithm of~\cite{Ahn1706:Information} cannot handle the sparse-weight regime, i.e., $\binom{n}{d}\alpha_n=\Theta(n)$.
In order to address this, we employ a preprocessing step prior to the 
spectral clustering step. 
It turns out this can handle the sparse regime; see  Lemma~\ref{lem:spec} for details.}

\rev{}{Another limitation of the original algorithm is related to its refinement step (to be detailed later).
The original refinement step is tailored for a specific model, which assumes binary-valued weights and two clusters (see Definition~\ref{def:sc}).
On the other hand, our new refinement step can be applied to the general case with weighted edges and $k$ clusters.
Further, the original refinement step involves iterative updates, and this is solely because our old proof holds only with such iterations. 
However, we observe via experiments that a single refinement step is always sufficient. 
By integrating a well-known sample splitting technique into our algorithm, we are able to prove that a single refinement step is indeed sufficient. }

\rev{}{Apart from the improvements above, we also propose a sketching algorithm for subspace clustering based on our new algorithm, and we show  that it outperforms existing schemes in terms of sample complexity as well as computational complexity.}

\subsubsection{Computer vision applications}
The weighted stochastic block model that we consider herein is well-fitted into computer vision applications such as geometric grouping and subspace clustering~\cite{ govindu2005tensor, agarwal2006higher, chen2009spectral}. The goal of such problems is to cluster a union of groups of data points where points in the same group lie on a common low-dimensional affine space. In these applications, similarity between a fixed number of data points reflects how well the points can be approximated by a low-dimensional flat. By viewing these similarities as the weights of edges in a hypergraph, one can relate it to our model.
Note that edges connecting the data points from the same low-dimensional affine space have larger weights  compared to other edges: See Section~\ref{sec:appl} for detailed discussion.

\subsubsection{Connection with low-rank tensor completion} Our model bears strong resemblance to the low-rank tensor completion. To see this, consider the following model: for each $e = \{i_1,i_2,i_3\} \in \mathcal{E}$, edge weight of $e$ is generated as $W_e=pX_e$ (where $X_e\sim \bern(\alpha_n)$) if $(i_1,i_2,i_3)$ are from the same cluster; $W_e=qX_e$ otherwise.
This model generates a weighted hypergraph, whose weights are either $p,q$, or $0$. Now, view each weight as an observation of an entry of a hidden tensor $\mathbf{T}$, whose entries $\mathbf{T}_{i_1i_2i_3} =p $ if $(i_1,i_2,i_3)$ are from the same cluster; $\mathbf{T}_{i_1i_2i_3} =q$ otherwise. Here, $0$ weight indicates that the entry is ``unobserved''. Then, the knowledge of hidden partition will directly lead to ``completion'' of unobserved entries. This way, one can draw a parallel between hypergraph clustering and the low-rank tensor completion.\footnote{Here, $\mathbf{T}$ is of rank at most $k$ since it admits a CP-decomposition~\cite{hitchcock1927expression} $\mathbf{T}=q \mathbf{1}^{\otimes 3} + \sum_{i=1}^k (p-q) (\mathbf{Z}_{*i})^{\otimes 3}$.}  
This connection allows us to compare our results with the guarantee in the tensor completion literature. For instance, the sufficient condition for vanishing estimation error, i.e., weak consistency, derived in~\cite{barak2016noisy} reads $\binom{n}{d}\alpha_n = \omega(n^{3/2}\log^4 n)$, while ours reads $\binom{n}{d}\alpha_n = \omega(n)$. This favors our approach. Moreover, a more interesting implication arises in computational aspects. Notice that a \naive~lower bound for tensor completion is\footnote{The number of free parameters defining a rank $k$, $d$-th order, $n$-dimensional tensor is $ndk$, which scales like $\Theta(n)$ when $d$ and $k$ are fixed.} $\binom{n}{d}\alpha_n =\Omega(n)$, and the tensor completion guarantee comes with an additional $\Omega(n^{1/2})$ factor to the lower bound. 
Actually this gap has not been closed in the literature, raising a question whether this \emph{information-computation gap} is fundamental. Interestingly, this gap does not appear in our result, hence hypergraph clustering can shed new light on the computational aspects of tensor completion. 
Recently, a similar observation has been made independently in~\cite{kim2017community} for spike-tensor-related models (see Sec. 4.3. therein).

\subsection{Paper organization}
Sec.~\ref{sec:model} introduces the considered model; 
in Sec.~\ref{sec:main}, our main results are presented along with some implications;
in Sec.~\ref{sec:pf}, we provide the proofs of the main theorems;
in Sec.~\ref{sec:dis}, we discuss as to how our results can be extended and adapted to other models; Sec.~\ref{sec:appl} is devoted to practical applications relevant to our model, and presents the empirical performances of the proposed algorithms; 
and in Sec.~\ref{sec:concl}, we conclude the paper with some future research directions.

\subsection{Notations}
Let $\mathbf{M}_{i*}$ ($\mathbf{M}_{*j}$) be the $i$th row (the $j$th column) of matrix $\mathbf{M}$. For a positive integer $n$, $[n]:= \{1,2,\ldots, n\}$. For a set $A$ and an integer $m$, $\binom{A}{m} := \{B \subset A \,:\, |B|=m \}$.
Let $\log(\cdot)$ denote the natural logarithm. 
Let $\ind \{\cdot\}$ denote the indicator function.
For a function $F: \mathcal{A} \to \mathcal{B} $ and $b\in  \mathcal{B}$,  $F^{-1}(b) :=\{i\in \mathcal{A}~:~ F(i) =b \}$. 

\section{The weighted stochastic block model}\label{sec:model}

\rev{}{We first remark that our definition of the weighted SBM is a generalization of the original model for graphs~\cite{jog2015information,xu2017optimal} to a hypergraph setting.}
For simplicity, we will focus on the following symmetric assortative  model in this paper.
In Sec.~\ref{sec:dis}, we generalized our results to a broader class of graph models.

\subsubsection{Model} Let $\mathcal{V}=[n]$ be the indices of $n$ nodes, and $\mathcal{E} :=\binom{[n]}{d}$ be the set of all possible edges of size $d$ for a fixed integer $d\geq 2$. 
Let $\Psi : \mathcal{V} \to [k]$ be the hidden partition function that maps $n$ nodes into $k$ groups for a fixed integer $k$. Equivalently, the membership function can be represented in a matrix form $\mathbf{Z} \in \{0,1\}^{n\times k}$, which we call \emph{the membership matrix}, whose $(i,j)$th entry takes $1$ if $j=\Psi(i)$ and $0$ otherwise. 
We denote by $n_i$ the size of   the  $j$th group for $j=1,2,\ldots, k$, i.e.,  $n_j := |\Psi^{-1}(j)|$. Let $n_{\min}:= \min_j n_j$ and $n_{\max}:= \max_j n_j$.
An edge $e=\{i_1,\ldots, i_d\}$ is \emph{homogeneous} if $\Psi(i_1)=\Psi(i_2)=\cdots = \Psi(i_d)$ and \emph{heterogeneous} otherwise.
We now formally define the weighted SBM.

\begin{definition}[The weighted SBM$(p,q,\alpha_n)$] \label{def:sSBM}
A random weight $W_e\in [0,1]$ is assigned to each edge $e$ independently\footnote{\rev{}{Our results hold as long as the weights are upper bounded by any fixed positive constant since one can always normalize the edge weights such that they are within $[0, 1]$. 
		The global upper bound on the edge weights are required for deriving our large deviation results (Lemmas~\ref{largedev} and \ref{machinery}) in the proof.}}: 
for homogeneous edges, $\ex[W_e] = p\alpha_n$; and 
for heterogeneous edges, $\ex[W_e] = q\alpha_n$.
\end{definition}
Note that the weighted SBM does not assume a specific edge weight distribution but only specifies the expected values. 
For instance, it can capture the case with a single location family distribution with different parameters as well as the case with two completely different weight distributions.

\begin{example}[The unweighted hypergraph case]
For homogeneous edges, $W_e \sim \bern(p \alpha_n)$; and for heterogeneous edges, $W_e \sim \bern(q \alpha_n)$.
This is an instance of the weighted SBM$(p, q, \alpha_n)$. 
When $d=2$, it captures the standard models such as planted multisection~\cite{mcsherry2001spectral} and the SBM~\cite{holland1983stochastic}.
\end{example}

\begin{example}[The weighted hypergraph case]
For homogeneous edges, $W_e \sim \bern(0.75)$; and for heterogeneous edges, $W_e \sim {\sf Unif}[0,1]$, a uniform distribution on $[0,1]$. 
This model can be seen as an instance of the weighted SBM$(0.75, 0.5, 1)$. 
\end{example}


\subsubsection{Performance metric}
Given $\{W_e\}_{e\in \mathcal{E}}$ \rev{}{and the number of clusters $k$}, we intend to recover a hidden partition $\Psi$ up to a permutation.
Formally, for any estimator $\Phi : [n] \to [k]$, we define the error fraction as $\err(\Phi):= \frac{1}{n}\min_{\Pi\in \mathcal{P}} |\{i:~ \Psi(i) \neq \Pi(\Phi(i)) \}|$, where $\mathcal{P}$ is the collection of all permutations of $[k]$.
We study three types of consistency guarantees~\cite{mossel2015consistency,abbe2017community}.
\begin{definition}[Recovery types] 
	\label{def:recovery}
	An estimator $\Phi$ is
	\begin{itemize}
	 \item  \emph{strongly consistent} if $\lim_{n\to \infty}\Pr(\err(\Phi)=0) = 1$;
	 \item \emph{weakly consistent} if  $\lim_{n\to \infty}\err(\Phi)= 0$ in prob.; and
	 \item is solving \emph{detection} if it outputs a partition which is more consistent relative to a random guess.\footnote{Here we provide an informal definition for simplicity. See Definition 7 in~\cite{abbe2017community} for the formal definition.}
	 \end{itemize}
\end{definition}

\section{Main results}\label{sec:main}

\subsection{Hypergraph Spectral Clustering} \label{hsc}

Hypergraph Spectral Clustering (HSC) is built upon the spectral relaxation technique~\cite{ghoshdastidar2015provable} and the spectral algorithms~\cite{feige2005spectral, coja2010graph,vu2014simple, guedon2016community,gao2015achieving,yun2014accurate,lei2015consistency}.
The first step of the algorithm is to {\bf compute the processed similarity matrix} whose entries represent similarities between pairs. To this end, we first compute \emph{the similarity matrix} $\mathbf{A}$, where $\mathbf{A}_{ij}=\sum_{ e:~\{i,j\}\subset e} W_e$ if $i\neq j$; $\mathbf{A}_{ij}=0$ if $i= j$.
This is inspired by the spectral relaxation technique in~\cite{ghoshdastidar2015provable}.
Next, we zero-out every row and column whose sum  is larger than  a certain threshold, constructing an output $\Az$, which we call \emph{the processed similarity matrix}. We then apply {\bf spectral clustering} to the processed similarity matrix. That is, we first find the $k$ largest eigenvectors  $\Uz \in \mathbb{R}^{n\times k}$ of $\Az$, and cluster $n$ rows of $\Uz$ using the approximate geometric $k$-clustering~\cite{matouvsek2000approximate}. 
Note that HSC is non-parametric, i.e., it does not require the knowledge of model parameters.
See Alg.~\ref{alg1} for the detailed procedure.

\rev{}{
	\begin{remark}
	The zeroing-out procedure, proposed in \cite{feige2005spectral} (see Sec.~3 therein), is used to remove outlier rows whose sums are much larger than the average.
	This is necessary since if such outliers exist, the eigenvector estimate will be biased, and hence the spectral clustering will also fail. 
	Note that this technique is widely adopted in various graph clustering algorithms~\cite{coja2010graph,chin2015stochastic,yun2014accurate}.
	\end{remark}}

\rev{}{The time complexity of HSC is $O\left(n^d\right)$. As each edge appears $2\binom{d}{2}$ times during the construction of the similarity matrix, this step requires $2\binom{d}{2}|\mathcal{E}|= O\left(n^d\right)$ time. The first $k$ eigenvectors can be computed via power iterations, which can be done within $O(kn^2\log n)$ time~\cite{boutsidis2015spectral}. Geometric $k$-clustering can be done in time $O(n (\log n)^k)$~\cite{matouvsek2000approximate}.}

\begin{algorithm}[tb]

	\caption{HSC}\label{alg1}
	\begin{algorithmic}[1]
		\State {\bfseries Input}: A weighted hypergraph $\mathcal{H} = ([n],\{W_e\}_{e\in \mathcal{E}})$, \rev{}{the number of clusters $k$}.
		\State {\bfseries Compute the processed  similarity matrix $\Az$}: Compute the similarity matrix $\mathbf{A}$ where $\mathbf{A}_{ij} = \sum_{ e:~\{i,j\}\subset e} W_e$ if $i
		\neq j$; and $\mathbf{A}_{ij} =0$ otherwise.
		Then, obtain $\Az$ by zeroing-out row $i$ (and the corresponding column) if $\sum_{j}\mathbf{A}_{ij} > \ctt \frac{1}{n}\sum_{i,j}\mathbf{A}_{ij}$, where $\ctt>0$ is a constant  depending only on $d$ (e.g., $\ctt=6$ when $d=2$).
		\State {\bfseries Apply spectral clustering to $\Az$}: Find $k$ largest eigenvectors of $\Az$, stack them side by side to obtain $\Uz \in \mathbb{R}^{n\times k}$, and cluster the rows of $\Uz$ using the approximate geometric $k$-clustering~\cite{matouvsek2000approximate} with an approximation rate $\epsilon>0$. 	
		\State {\bfseries Output:} $\phsc(i)=$ cluster index of the $i$th row.
	\end{algorithmic}
\end{algorithm}

\subsection{Hypergraph Spectral Clustering with Local Refinement}

Our second algorithm consists of two stages: HSC and local refinement.
The HSCLR algorithm is inspired by a similar refinement procedure, which has been proposed for the graph case~\cite{abbe2016exact,abbe2015community}.
The algorithm begins with randomly splitting edges into two sets $\mathcal{E}_1$  and $\mathcal{E}_2$.~For small $\beta>0$, we assign each edge to $\mathcal{E}_1$ independently with probability $\beta$. $\mathcal{E}_2$  is the complement of $\mathcal{E}_1$.
Then, we run HSC  on $\mathcal{H}_1 = ([n], \{W_e\}_{e\in\mathcal{E}_1})$. 
 Next, we do local refinement with $\mathcal{E}_2$. 
For $i\in [n]$ and $j\in [k] $, define $\nei (j)$ to be the set of edges ( $\in\mathcal{E}_2$)  which connect node $i$ with $d-1$ nodes from $\phsc^{-1}(j)$, i.e.,  $\nei(j):=\left\{e \in \mathcal{E}_2~:~ i\in e,~ (e\setminus \{i\})\subset \phsc^{-1}(j)   \right\}$.
Then, for each $i\in [n]$, we update $\phsc(i)$ with
\begin{align}
\arg\max_{j\in [k]} \frac{1}{|\nei (j)|} \sum_{e\in \nei (j)} W_e. \label{refinerule}
\end{align}
That is, the refinement step first measures the \emph{fitness} of each node with respect to  different  clusters, and updates the cluster assignment of each node accordingly.
Note that HSCLR is also non-parametric.
 See Alg.~\ref{alg2} for the detailed procedure.

\rev{}{The time complexity of HSCLR is $O\left(n^d \right)$. For each node $i$, the local refinement requires $\sum_{j=1}^{k}|\nei(j)|$ flops, which is bounded by $k|\mathcal{E}_i|$, where $|\mathcal{E}_i|$ is the number of edges containing node $i$.
As $\sum_i |\mathcal{E}_i| = d |\mathcal{E}|$, the local refinement step can be done within $O\left(|\mathcal{E}| \right)$ time.  }
\begin{remark}
	HSCLR is inspired by the recent paradigm of solving non-convex problems,  which first approximately estimates the solution, followed by some local refinement. 
	This two-stage approach has been applied to a variety of contexts, including matrix completion~\cite{keshavan2010matrix,jain2013low}, phase retrieval~\cite{netrapalli2013phase,candes2015phase}, robust PCA~\cite{yi2016fast}, community recovery~\cite{abbe2016exact,chen2016community}, EM-algorithm~\cite{balakrishnan2017}, and rank aggregation~\cite{chen2015spectral}.
\end{remark}

\begin{algorithm}[tb]
	\caption{HSCLR}\label{alg2}
	\begin{algorithmic}[1]
		\State {\bfseries Input}: A weighted hypergraph $\mathcal{H} = ([n],\{W_e\}_{e\in \mathcal{E}})$, \rev{}{the number of clusters $k$}, and sample splitting rate $\beta>0$.
		\State Randomly split $\mathcal{E}$: for small enough $\beta>0$, include each edge of $\mathcal{E}$ in $\mathcal{E}_1$ independently with probability $\beta$. Denote by $\mathcal{E}_2$ the complement of $\mathcal{E}_1$.
		\State {\bfseries Apply Hypergraph Spectral Clustering to $\mathcal{H}_1 = ([n], \{W_e\}_{e\in\mathcal{E}_1})$} to yield an estimate $\phsc$.
		\State {\bf Local refinement}: for $i=1,2,\ldots, n$, $\phsclr(i) = \arg\max_{j\in [k]} \frac{1}{|\nei (j)|} \sum_{e\in \nei (j)} W_e$.
		\State {\bfseries Output}: $\phsclr$.
	\end{algorithmic}
\end{algorithm}

\subsection{Theoretical guarantees}
\begin{theorem} \label{thm:main}
Let $\phsc$ be the output of $HSC$.
Suppose that $\frac{n_{\max}}{n_{\min}}=O(1)$. 
Then, there exist constants $\cweakone,\cweaktwo>0$ (where $\cweaktwo$ depends on $p$ and $q$) such that if $\binom{n}{d}\alpha_n  \geq \cweakone n$,
	then,
	\begin{align}\label{eq:thm1}
		\err(\phsc) \leq \cweaktwo k\frac{n^3}{d(d-1)n_{\min}^2 \binom{n}{d}\alpha_n }
	\end{align}
	w.p. $1-O(n^{-1})$, provided that $\cweaktwo k\frac{n^3}{d(d-1)n_{\min}^2 \binom{n}{d}\alpha_n }<1$. 
\end{theorem}
\begin{IEEEproof}
	See Sec.~\ref{pf:thm1}.
\end{IEEEproof}

\rev{}{Note that when $d=2$, Thm.~\ref{thm:main} recovers \cite[Thm.~6]{gao2015achieving}.}

	\begin{remark}
		We remark a technical challenge that arises in proving Thm.~\ref{thm:main} relative to the graph case. Actually, the key step in the proof is to derive the sharp concentration bound on a certain matrix spectral norm (to be detailed later). But the bounding technique employed in the graph case does not carry over to the hypergraph case, as the matrix has strong dependencies across entries. We address this challenge by developing a delicate analysis that carefully handles such dependencies. See Remark~\ref{rmk2} in Sec.~\ref{sec:pf} for details.
	\end{remark}

\begin{corollary}[Detection]
	Suppose that $\frac{n_{\max}}{n_{\min}}=O(1)$. There exists a constant $\cdetec$ depending on $p,q$ and $k$ such that HSC solves detection if $\binom{n}{d}\alpha_n \geq \cdetec \cdot n$.
\end{corollary}
\begin{IEEEproof}
	In Thm.~\ref{thm:main}, $\err(\phsc)=O(1/c)$  when $\alpha_n$ satisfies $\binom{n}{d}\alpha_n\geq cn$ for sufficiently large $c>0$.
\end{IEEEproof}
\begin{remark}
	 We compare our algorithm to the one proposed in \cite{angelini}.
		To compare, we first note that in the graph case, the threshold for detection~\cite{decelle} is achieved by new methods based on the \emph{non-backtracking operator}~\cite{Krzakala24122013,bordenave,abbe2015detection}. 
		In~\cite{Krzakala24122013}, the spectral analysis based on a plain adjacency matrix is shown to fail, while the one based on the non-backtracking operator succeeds.
		Recently, it is shown that the non-backtracking based approach can be extended to the hypergraph case, and it is empirically observed to outperform a spectral method that is similar to HSC except the preprocessing step~\cite{angelini}.
\end{remark}

\begin{corollary}[Weak consistency] \label{weakc}	Suppose that $\frac{n_{\max}}{n_{\min}}=O(1)$.  HSC is weakly consistent if $\binom{n}{d}\alpha_n = \omega(n)$.
\end{corollary}
\begin{IEEEproof}
By~\eqref{eq:thm1}, $\binom{n}{d}\alpha_n= \omega(n) \Rightarrow \err(\phsc)=o(1)$.
\end{IEEEproof}
 
\begin{remark}
When specialized to weighted stochastic block model, the weak consistency guarantee of~\cite{JMLR:v18:16-100} becomes $\binom{n}{d}\alpha_n = \Omega(n(\log n)^{2})$, which comes with an extra poly-logarithmic factor gap to ours.
\end{remark}

The following theorem provides the theoretical guarantee of HSCLR. See Sec.~\ref{pf:thm2} for the proof.

\begin{theorem}[Strong consistency]\label{thm:main2}
		Suppose that $\frac{n_{\max}}{n_{\min}}=O(1)$. Then, HSCLR with sampling rate\footnote{\rev{}{We note that $\beta$ can be chosen arbitrarily as long as $\beta = o(1)$ and $\beta = \omega (1/ \log n)$. See Sec.~\ref{pf:thm2} for detail.}} $\beta = \frac{\log\log n}{\log n}$ is strongly consistent provided that  for any $\epsilon>0$, \begin{align}
		\frac{(p-q)^2}{p}\binom{n}{d}\alpha_n \geq (8 +\epsilon)\frac{\left(n/n_{\min}\right)^{d-1}}{d} n\log n\,.
		\end{align}
		
	\end{theorem}

\rev{}{\begin{remark}
		We remark that Thm.~\ref{thm:main2} characterizes the performance of our \emph{non-parametric} algorithm for any hypergraphs with (bounded) \emph{real-valued} weights. 
		Hence, one may obtain a tighter threshold and a \emph{parametric} algorithm by focusing on a more specific hypergraph model. 
		For instance, in~\cite{chien2018minimax}, Chien et al. derive a tighter bound for the binary weight case. As a concrete example, when $d = 3$ and $k = 2$ with two equal-sized clusters, the sufficient condition of Thm~4.1 in \cite{chien2018minimax} reads
		$(\sqrt{p}-\sqrt{q})^2\binom{n}{3}\alpha_n \geq \frac{1}{12} n\log n,$
		while that of Thm.~\ref{thm:main2} reads $\frac{(p-q)^2}{p}\binom{n}{3}\alpha_n \geq \frac{32}{3} n\log n$.
\end{remark}}
Indeed, by leveraging recent works on phase transition of random hypergraphs~\cite{coja2007counting,COOLEY2015569}, we can prove the order-wise optimality of our algorithms for the binary-valued edge case. 
\begin{proposition} \label{optimal}
	For the binary-valued edge case, there is no estimator which
	\begin{itemize}
		\item solves detection when $\binom{n}{d}\alpha_n = o(n)$
		\item is weakly consistent when $\binom{n}{d}\alpha_n = O(n)$; and
		\item is strongly consistent when $\binom{n}{d}\alpha_n = o(n\log n)$. 
	\end{itemize}
\end{proposition}
\begin{IEEEproof}
If $\binom{n}{d}\alpha_n = o(n)$, the fraction of isolated nodes approaches $1$, hence detection is infeasible. 
In~\cite{coja2007counting}, the authors show that if  $\binom{n}{d}\alpha_n = \Theta(n)$, 
there is no connected component of size $(1-o(1))n$, implying that weak consistency is infeasible.
Lastly, \cite{COOLEY2015569} shows that  $\binom{n}{d}\alpha_n > c\cdot n\log n$ for some constant $c>0$ is required for connectivity, a necessary condition for strong consistency.
\end{IEEEproof}

\section{Proofs}\label{sec:pf}

\subsection{Proof of Theorem~\ref{thm:main}} 
\label{pf:thm1}
We first outline the proof.
Proposition~\ref{prop} asserts that spectral clustering finds the exact clustering if $\ex[\mathbf{A}]$ is available instead of $\mathbf{A}^0$.
We then make use of Lemma~\ref{lem:sc} to bound the error fraction in terms of $\|\Az - \ex[\mathbf{A}]\|$.
Finally, we derive a new concentration bound for the above spectral norm, and combine it with Lemma~\ref{lem:sc} to prove the theorem.

Consider two off-diagonal entries $\mathbf{A}_{i,j}$ and $\mathbf{A}_{i',j'}$ such that $\Psi(i)=\Psi(i')$ and $\Psi(j)=\Psi(j')$. One can see from the definition that $\mathbf{A}_{i,j}$ is statistically identical to $\mathbf{A}_{i'j'}$, so $\ex [\mathbf{A}_{ij}] = \ex [\mathbf{A}_{i'j'}]$. Hence, by defining a $k\times k$ matrix $\mathbf{B}$ such that 
$\mathbf{B}_{\ell,m} = \ex [\mathbf{A}_{ij}]$, where $i\in \Psi^{-1}(\ell), ~j\in \Psi^{-1}(m)$, for some $i\neq j$, one can verify that 
$\mathbf{P} :=\mathbf{Z} \mathbf{B}\mathbf{Z}^T$
coincides with $\ex [\mathbf{A}]$ except for the diagonal entries.
Our model implies that the diagonal entries of $\mathbf{B}$ are strictly larger than its off-diagonal entries, so $\mathbf{B}$ is of full rank.
\begin{proposition} \label{prop}
	(Lemma~2.1 in~\cite{lei2015consistency}) Consider $\mathbf{B}\in \mathbb{R}^{k\times k}$ of full rank and the membership matrix $\mathbf{Z}\in\{0,1\}^{n\times k}$. Let $\mathbf{P} = \mathbf{Z} \mathbf{B}\mathbf{Z}^T$. Then the matrix $\mathbf{U} \in \mathbb{R}^{n\times k}$ whose columns are the first $k$ eigenvectors of $\mathbf{P}$ satisfies:$\mathbf{U}_{i*} = \mathbf{U}_{j*}$ whenever $\Psi(i) = \Psi(j)$; $\mathbf{U}_{i*}$ and $\mathbf{U}_{j*}$ are orthogonal whenever $\Psi(i) \neq \Psi(j)$.	
	In particular, a clustering algorithm on the rows of $\mathbf{U}$ will exactly output the hidden partition.
\end{proposition}

Proposition~\ref{prop} suggests that spectral clustering  successfully finds  $\Psi$  if $\mathbf{P}$ is available. We now turn to the case  where  $\Az$ is available instead of $\mathbf{P}$.  It is developed in \cite{lei2015consistency} a general scheme to prove error bounds for spectral clustering under an assumption that $k$-clustering step outputs a ``good'' solution. 
To clarify the meaning of ``goodness'', we formally describe the $k$-means clustering problem. 
\begin{definition}[$k$-means clustering problem]
The goal is to cluster the rows of an $n\times k$ matrix $\mathbf{U}$. Define the cost function of a partition $\Phi :[n] \to [k]$ as $\cost(\Phi) = \sum_{j=1}^k \var(\Phi^{-1}(j))$,
where 
$\var(\mathcal{A}) = \sum_{i\in \mathcal{A}} \left\|\mathbf{U}_{i*} -  \frac{1}{|\mathcal{A}|} \sum_{\ell \in \mathcal{A}} \mathbf{U}_{\ell}  \right\|^2\,$.
We say $\Phi$ is $(1+\epsilon)$-approximate if 
$\cost(\Phi) \leq (1+\epsilon ) \min_{\Phi' : [n]\to [k]} \cost(\Phi')\,$.
\end{definition}
We now introduce the general scheme to prove error bounds, formally stated in the following lemma. 
\begin{lemma} \label{lem:sc}
	Assume that $\mathbf{P}$ is defined as in Proposition~\ref{prop} and $\sigma_{\min}(\mathbf{P})$ is the smallest non-zero singular value of $\mathbf{P}$. Let $\mathbf{M}$ be any symmetric matrix and $\mathbf{U}\in\mathbb{R}^{n\times k}$ be the $k$ largest eigenvectors of $\mathbf{M}$. Suppose a $(1+\epsilon)$-approximate solution $\Phi$ for a constant $\epsilon>0$.
	Then, for some $\cccc>0$,
$\err(\Phi) \leq \cccc k(1+\epsilon)\frac{\|\mathbf{\mathbf{M}-\mathbf{P}}\|^2}{\sigma_{\min}(\mathbf{P})^2}$,
	provided that $\cccc k(1+\epsilon)\frac{\|\mathbf{\mathbf{M}-\mathbf{P}}\|^2}{\sigma_{\min}(\mathbf{P})^2} \leq 1$.
\end{lemma}
\begin{IEEEproof}
	We refer to \cite{lei2015consistency} for the proof.
\end{IEEEproof}

Thm.~1.2. in~\cite{matouvsek2000approximate} implies that a $(1+\epsilon)$-approximate solution can be found using the approximate geometric $k$-clustering.\footnote{\rev{}{Note that this result holds only for a fixed $k$~\cite{matouvsek2000approximate}.}}
Hence, the above lemma implies that one needs to bound $\|\Az - \mathbf{P}\|$ in order to analyze the error fraction of the spectral clustering.
Our technical contribution lies mainly in deriving such concentration bound, formally stated below.
\begin{lemma}\label{lem:spec}
	There exist constants $\ctt$ (depending only on $d$), $\cddd,\ceee>0$ such that the processed similarity matrix $\Az$ with constant $\ctt$ (see Alg.~\ref{alg1}) satisfies
	$\|\mathbf{A}^0 - \mathbf{P}\| \leq \cddd\sqrt{n  \binom{n-2}{d-2}\alpha_n}$
	 with probability exceeding $1-O(n^{-1})$, provided that $n \binom{n-2}{d-2}p\alpha_n \geq \ceee$.
\end{lemma}
\begin{IEEEproof}
	See Appendix~\ref{pf:spec}.
\end{IEEEproof}


\rev{}{Note that this lemma holds for a fixed $d$.}
We now conclude the proof with these lemmas. Let $\mu:= \binom{n-2}{d-2} \alpha_n $. We first estimate $\sigma_{\min}(\mathbf{P})$.
\begin{claim}
	$\sigma_{\min}(\mathbf{P}) \leq \cfff n_{\min} \mu $ for some constant $\cfff>0$.
\end{claim}
\begin{IEEEproof}
 By definition, 
$\mathbf{P}= (\mathbf{Z} \Delta^{-1}) \Delta \mathbf{B} \Delta^T (\mathbf{Z} \Delta^{-1})^T$, where $\Delta = \text{diag}(\sqrt{n_1},\sqrt{n_2},\ldots, \sqrt{n_k}  ) $. Since the columns of $\mathbf{Z} \Delta^{-1}$ are orthonormal, $\sigma_{\min}(\mathbf{P}) =\sigma_{\min}(\Delta \mathbf{B} \Delta^T)$.
One can show that $\sigma_{\min}(\Delta \mathbf{B} \Delta^T) \geq \sigma_{\min}(\Delta)^{2} \sigma_{\min}(\mathbf{B})$.
Hence, $\sigma_{\min}(\mathbf{P}) \geq \sigma_{\min}(\Delta)^2 \sigma_{\min}(\mathbf{B}) = n_{\min} \sigma_{\min}(\mathbf{B})$.
 Hence, we calculate $\sigma_{\min}(\mathbf{B})$.
 By the definition of $\mathbf{B}$,
\begin{align*}
\mathbf{B}_{\ell m} &= \begin{cases}p\alpha_n \binom{n_\ell -2}{d-2} + q\alpha_n \left[\binom{n-2}{d-2} -\binom{n_\ell-2}{d-2}\right] &\text{if}~\ell =m,   \\ q\alpha_n \binom{n-2}{d-2}&\text{if}~\ell \neq m \end{cases}
\\
&=\mu \cdot\begin{cases}  \frac{\binom{n_{\ell}-2}{d-2}}{\binom{n-2}{d-2}}(p-q) +q  &\text{if}~\ell =m,   \\   q&\text{if}~\ell \neq m. \end{cases}
\end{align*}
Thus, $\mathbf{B} = \mu \cdot q\mathbf{1} \mathbf{1}^T + \mu \cdot \text{diag}(f_1,f_2,\ldots, f_\ell)$, where $f_\ell:=\frac{\binom{n_{\ell}-2}{d-2}}{\binom{n-2}{d-2}}(p-q)$.
As $n_{\max}/n_{\min} = O(1)$, each $f_\ell$ converges to a positive constant, implying that $\sigma_{\min}(\mathbf{B})= \Theta(\mu) $.
\end{IEEEproof}
By Lemma~\ref{lem:spec} and the above claim, 
$\cccc k(1+\epsilon)\frac{\|\Az-\mathbf{P}\|^2}{\sigma_{\min}(\mathbf{P})^2} \leq \frac{\cccc k(1+\epsilon)\cddd ^2}{\cfff^2}\frac{n}{n_{\min}^2 \mu}$ holds w.p. $1-O(n^{-1})$ for $n\mu \geq \ceee$.
Choosing $\caa =\frac{\ceee }{d(d-1)}$, $\caaa = \frac{\cccc k(1+\epsilon)\cddd^2}{\cfff^2} $ completes the proof.

\begin{remark} \label{rmk2}(Technical novelty relative to the graph case):
	Indeed, proving the sharp concentration of a spectral norm  has been a key challenge in the spectral analysis~\cite{friedman1989second,feige2005spectral}. 
	While most bounds developed hinge upon the independence between entries\footnote{For instance, the most studied model, called the Wigner matrix, assumes independence among entries. See~\cite{tao2012topics} for more details.}, the matrix $\mathbf{A}$ in HSC has strong dependencies across entries due to its construction. 
	For instance, the entries $\mathbf{A}_{12}$ and $\mathbf{A}_{13}$ both have a term $W_e$ for any edge $e$ of the form $\{1,2,3,j_4,j_5,\ldots, j_d \}$, hence sharing $\binom{n-3}{d-3}$ many terms.
	
	One approach to handle this dependency is to use matrix Bernstein inequality~\cite{tropp2012user} on the decomposition $\mathbf{A} =\sum_{e\in \mathcal{E}}W_e \mathbf{S}_e$,
	where  $S_e := \sum_{\substack{i,j\in e\\ i\neq j}} \mathbf{e}_i \mathbf{e}_j^T$. See~\cite{ghoshdastidar2015consistency,JMLR:v18:16-100}. However, this approach provides a bound which comes with an extra $\sqrt{\log n}$ factor  relative to the bound in Lemma~\ref{lem:spec}, resulting in a suboptimal consistency guarantee as described in Sec.~\ref{relatedwork}.  
	
	 Another approach is a combinatorial method~\cite{friedman1989second}, which counts the number of edges between subsets.
	 The rationale behind this method is as follows. 
	 From the definition of the spectral norm, one needs to bound the quantity $\mathbf{x}^T(\Az-\mathbf{P})\mathbf{x}$ for any vector $\mathbf{x}$.
	It turns out that this quantity has a close connection to the number of (hyper)edges between two subsets in a random (hyper)graph. 
	 For instance, $\mathbf{1}_A^T \mathbf{A} \mathbf{1}_B$ is precisely the number of edges between $A$ and $B$. 
	 
	  Indeed, a technique for estimating the number of hyperedges between two arbitrary subsets is developed in~\cite{jain2014provable}. 
	 Using this method, however, one may only obtain a suboptimal guarantee, which is $\binom{n}{d}\alpha_n=\Omega(n^{1.5})$.
On the other hand, we show via our analysis that the order-optimal guarantee can be obtained by improving the standard combinatorial method. See Appendix~\ref{pf:spec}.
\end{remark}

\subsection{Proof of Theorem~\ref{thm:main2}} \label{pf:thm2}
We first outline the proof. 
Using the union bound, we show that it is sufficient to prove  $\Pr(\phsclr(i) = j) = o(n^{-1})$ for all $1\leq i \leq n$ and $j \neq \Psi(i)$.
We then consider the following events to bound this error probability.
The first event is that the average edge weight of the edges between the true community $\Psi(i)$ and node $i$ is less than a certain threshold, and the other one is that the average edge weight of the edges between the wrong community $j$ and node $i$ is greater than the certain threshold. 
We will first show that if the misclassification event occurs, at least one of these two events must occur.
Thus, we bound the error probability by bounding those of these two events using Lemma~\ref{largedev} and Lemma~\ref{marginal}, respectively.

We consider the boundary case $\binom{n}{d}\alpha_n  = \Theta(n \log n)$. 
\rev{}{As $\beta \binom{n}{d}\alpha_n = \Theta(n\log\log n) = \omega(n)$, Corollary~\ref{weakc} guarantees that $\phsc$ is weakly consistent.}
Without loss of generality, assume that the identity permutation is equal to $\argmin_{\Pi\in \mathcal{P}} |\{i: \Psi(i) \neq \Pi(\phsc(i)) \}|$.
Then, $
\frac{|\phsc^{-1}(j) \cap \Psi^{-1}(j)|}{|\phsc^{-1}(j)|} > 1-\gamma\label{almostcorr}$, i.e., at least $1-\gamma$ fraction of the nodes that are classified as in community $j$ are correctly classified. 
The second stage of HSCLR refines the output of the first stage $\phsc$, resulting in $\phsclr$. 
By the union bound, we have $\Pr(\err(\phsclr)\neq 0) \leq \sum_{i=1}^{n} \sum_{j\neq \Psi(i)} \Pr(\phsclr(i) = j)$.
Since the total number of summands is $\Theta(n)$, if $\Pr(\phsclr(i) = j) = o(n^{-1})$ for all $1\leq i \leq n$ and $j \neq \Psi(i)$, then $\Pr(\err(\phsclr)\neq 0) = o(1)$.

By the refinement rule~\eqref{refinerule}, $\Pr(\phsclr(i) = j) \leq \Pr\left( \frac{\sum_{e\in \nei (\Psi(i))} W_e}{|\nei (\Psi(i))|} < \frac{\sum_{e\in \nei (j)} W_e}{|\nei (j)|}   \right)$.
For any real numbers $(a,b,t)$, $\left[ a \geq b \right] \supset \left[a \geq t\right] \cap \left[t \geq b\right]$ holds. 
By taking complements of both sides, we have $\left[ a<b\right] \subset  \left[a<t\right] \cup \left[t<b\right]$. 
Therefore, by the union bound, $P(a < b) \leq P(a < t) + P(t < b)$ holds for any $(a,b,t)$.
Applying this bound, we have
\begin{align}
&\Pr\left( \frac{\sum_{e\in \nei (\Psi(i))} W_e}{|\nei (\Psi(i))|}  < \frac{\sum_{e\in \nei (j)} W_e}{|\nei (j)|}   \right) \label{basepr} \\
&\leq \underbrace{\Pr\left( \frac{\sum_{e\in \nei (\Psi(i))} W_e}{|\nei (\Psi(i))|}  < \frac{p+q}{2} \alpha_n \right)}_{R_1}\\
&+\underbrace{\Pr\left( \frac{p+q}{2}\alpha_n < \frac{\sum_{e\in \nei (j)} W_e}{|\nei (j)|} \right)}_{R_2}\,, 
\end{align}

We first interpret $R_1$ and $R_2$.
For illustration, assume that $\phsclr$ coincides with $\Psi$. 
Under this assumption, observe that
 $\frac{\sum_{e\in \nei (\Psi(i))} W_e}{|\nei (\Psi(i))|}$ is equal to the average edge weight of the homogeneous edges within community $\Psi(i)$.
Since the expected value of this term is $\frac{p}{2}\alpha_n$, one can show that the term $R_1$ vanishes.
Similarly, $\frac{\sum_{e\in \nei (j)} W_e}{|\nei (j)|} $ is the average weight of the edges connecting $i$ and the other nodes in community $j$. 
Since these edges are heterogeneous, $R_2$ also vanishes.

Indeed, as $\err(\phsc)$ is not exactly zero, but an arbitrarily small constant, the above interpretation is not precise.
In what follows, we show that $R_1$ and $R_2$ vanish as well for the case.

We begin with bounding $R_1$.
Denote by $\mathcal{E}_h$ the set of all homogeneous edges. 
Recall that edges in $\nei (\Psi(i))$, except $O(\gamma)$ fraction, are homogeneous, so $|\nei (\Psi(i))\cap \mathcal{E}_h| =(1-O(\gamma))|\nei (\Psi(i))|$. By restricting the range of summation,
$R_1 \leq \Pr\left(  \sum_{e\in \nei (\Psi(i))\cap \mathcal{E}_h} W_e < \frac{p+q}{2}\alpha_n |\nei (\Psi(i))|\right)$.
 Note that $W_e$'s are not restricted to Bernoulli random variables.
	By tweaking the proof of conventional large deviation results~\cite{alon2004probabilistic} for Bernoulli variables, we obtain the following: 
\begin{lemma} \label{largedev}
	Let $S$ be the sum of $m$ mutually independent random variables taking values in $[0,1]$. For any $\delta>0$, we have $\Pr\left( S > (1+\delta) \ex [S] \right) \leq 	\exp\left( -\frac{\delta^2}{2+\delta}\ex [S] \right)$ and $\Pr\left( S < (1-\delta) \ex [S] \right) \leq \exp\left( -\frac{\delta^2}{2}\ex [S] \right)$.
\end{lemma}
\begin{IEEEproof}
See Appendix~\ref{pf:largedev}.
\end{IEEEproof}

As $\ex [\sum_{e\in \nei (\Psi(i))\cap \mathcal{E}_h} W_e ]= (1-O(\gamma))p\alpha_n |\nei (\Psi(i))|$,
\begin{align*}
\frac{\frac{p+q}{2}\alpha_n |\nei (\Psi(i))|  }{\ex [\sum_{e\in \nei (\Psi(i))\cap \mathcal{E}_h} W_e]} -1 &= (1+O(\gamma)) \frac{q-p}{2p},
\end{align*}
 so Lemma~\ref{largedev}$-2)$ with $\delta = (1+O(\gamma)) \frac{p-q}{2p}$ gives 
\begin{align}
&R_1 \leq \exp\left( -\frac{(p-q)^2}{8p}\alpha_n(1+O(\gamma))|\nei (\Psi(i))| \right).  \label{int:1}
\end{align}
Next we consider $R_2$. Again, edges in $\nei (j)$, except $O(\gamma)$ fraction, are heterogeneous, so $|\nei (j)\cap \mathcal{E}_h^c| =(1-O(\gamma))|\nei (j)|$. The following lemma says that the contribution due to the $O(\gamma)$ fraction of edges is marginal:
\begin{lemma} \label{marginal}
	For sufficiently small $\gamma>0$,\begin{align}
	\Pr\left( \sum_{e\in \nei (j)\cap \mathcal{E}_h} W_e > \frac{p\alpha_n |\nei (j)| }{\sqrt{\log(1/\gamma)}} \right) =o(n^{-1})\,.
	\end{align}
\end{lemma}
\begin{IEEEproof}
See Appendix~\ref{pf:marginal}.
\end{IEEEproof}

Hence, we focus on heterogeneous edges only.
Making a similar argument as above, the bound in Lemma~\ref{largedev} becomes
\begin{align}
&R_2 \leq \exp\left( - \frac{\frac{(p-q)^2}{4q^2}}{2+ \frac{p-q}{2q}} q\alpha_n(1+O(\gamma))|\nei (j)|  \right)\\
&\overset{(a)}{\leq}\exp\left( - \frac{1}{8}\frac{(p-q)^2}{p} \alpha_n(1+O(\gamma))|\nei (j)|  \right)\,,  \label{int:2}
\end{align}
where ($a$) follows since $\frac{1}{2+\frac{p-q}{2q}} = \frac{1}{\frac{3}{2}+\frac{p}{2q}} = \frac{1}{\left(\frac{3}{2}\frac{q}{p}+\frac{1}{2}\right)\frac{p}{q}} \geq \frac{1}{2\frac{p}{q}}$.

Since $\frac{(p-q)^2}{p}\binom{n}{d}\alpha_n \geq (8 +\epsilon)\frac{\left(n/n_{\min}\right)^{d-1}}{d} n\log n$, a straightforward calculation yields
$\frac{1}{8}\frac{(p-q)^2}{p} \alpha_n(1+O(\gamma))\binom{n_{\min}-1}{d-1} \geq \left(1+\frac{1}{16}\epsilon\right) \log n$, for sufficiently large $n$.
Thus, $R_1$ and $R_2$ are both $o(n^{-1})$ from \eqref{int:1} and \eqref{int:2}.

\section{Discussion}\label{sec:dis}

 We have shown that our algorithms can achieve the order-optimal sample complexity for all different recovery guarantees under a symmetric block model.
In this section, we show that our main results indeed hold for a broader class of block models. 
We also show that HSCLR can achieve the sharp recovery threshold for a certain SBM model.

\subsection{Extensions}
For the graph case~\cite{abbe2017community}, a fairly general model, which subsumes as a special case the asymmetric SBM, has been investigated. Here we extend our model to one such model but in the context of hypergraphs.  Specifically, we consider the following asymmetric weighted SBM. 
\begin{definition}[The asymmetric weighted SBM] \label{def:asym_wsbm} 
 Let $\{p_e\}_{e\in\mathcal{E}}$ be constants such that $p_e > p_{e’}$ holds for any homogeneous edge $e$ and heterogeneous edge $e’$.
 A random weight is assigned to each edge independently as follows:
  For each edge $e \in \mathcal{E}$, 
$\ex[W_e] =p_e \alpha_n$. Notice that this reduces to the condition of $p > q$ in the symmetric setting. 
\end{definition}

We find that our main results stated in Thm.~\ref{thm:main} and \ref{thm:main2} readily carry over the above asymmetric setting. The key rationale behind this is that our spectral clustering guarantee hinges only upon the full-rank condition on $\mathbf{B}$ (see Sec.~\ref{hsc} for the definition). Here, what one can easily verify is that the condition above implies the full-rank condition, and hence our results hold even for the asymmetric setting. The only distinction here is that the constants that appear in the theorems depend now on $p_e$'s.
Similarly, our technique can cover \emph{disassortative} SBM in which heterogeneous edges have larger weights than homogeneous edges.

\begin{definition}[The symmetric disassortative weighted SBM] \label{def:sym_dissbm}
	In Definition~\ref{def:sSBM}, we assume instead that $0<p<q<1$.
\end{definition}

Another prominent instance is the planted clique model.
\begin{definition}[The planted clique model] \label{def:clique}
	  Fix $s$-subset of nodes $C$ ($s\leq n$). Consider a random hypergraph in which every $d$-regular edge $e = \{i_1,i_2,\ldots, i_d \}$ appears with probability $1$ if $e \subseteq C$ or $\frac{1}{2}$ otherwise.
\end{definition}
In this model, one wishes to detect the hidden subset $C$, which is called the \emph{clique}. 
Following a similar analysis with a different notion of error fraction, one can show that the clique can be detected if $s\geq c^*\cdot \sqrt{n}$ for some constant $c^*$, which is consistent with the well-known result for $d=2$~\cite{alon1998finding}.

\subsection{Sharpness} \label{subsec:sharp}

Recently, sharp thresholds on the fundamental limits are characterized in the graph case~\cite{abbe2015community,abbe2015detection,abbe2016exact,gao2015achieving,zhang2016minimax}.
In contrast, such a tight result has been widely open in the hypergraph case. 
A notable exception is our companion paper~\cite{ahn2016community} which studies a special case of the weighted SBM (considered herein), in which weights are \emph{binary}-valued.

\begin{definition}[Generalized Censored Block Model with Homogeneity Measurements~\cite{ahn2016community}] \label{def:sc}   Let $\theta \in(0,1/2)$ be a fixed constant. Assume that $k=2$ and denote erasure by $\era$.
If the edge $e$ is homogeneous, $W_e = 1$ w.p. $\alpha_n (1-\theta)$, $W_e = 0$ w.p. $\alpha_n \theta$, and $W_e = \era$ w.p. $1-\alpha_n$.
Otherwise, $W_e = 1$ w.p. $\alpha_n \theta$, $W_e = 0$ w.p. $\alpha_n (1-\theta)$, and $W_e = \era$ w.p. $1-\alpha_n$. 
\end{definition}

The information-theoretic limit for strong consistency has been characterized under this model, formally stated below.
\begin{proposition} \label{prop2}
	(Thm.~1 in \cite{ahn2016community}) Under the model in Definition~\ref{def:sc}, the maximum likelihood estimator is strongly consistent for any given hidden partition $\Psi$ if $\binom{n}{d}\alpha_n \geq (1+\epsilon)\frac{2^{d-2}}{d} \frac{n \log n}{(\sqrt{1-\theta}-\sqrt{\theta})^2}$ for any constant $\epsilon>0$.
	Conversely, if $\binom{n}{d}\alpha_n \leq (1-\epsilon)\frac{2^{d-2}}{d} \frac{n \log n}{(\sqrt{1-\theta}-\sqrt{\theta})^2}$, no algorithm can be strongly consistent for any given hidden partition $\Psi$.
\end{proposition}
Using our results, we can show that there is no computational barrier under this model. 
We now state the theorem, deferring the proof to Appendix~\ref{pf:info}.
\begin{theorem} \label{thm:info}
	HSCLR\footnote{Indeed, there should be some minor tweaks to make HSCLR better adapted to this model. See Appendix~\ref{pf:info} for details. }  achieves the information-theoretic limits characterized in Proposition~\ref{prop2}.
\end{theorem}

\section{Application} \label{sec:appl}
 In this section, based on our algorithms, we design a sketching algorithm for subspace clustering.

\begin{figure*}[t!]
	\centering
	\begin{subfigure}[c]{0.20\textwidth}
		\centering
		\includegraphics[width=\columnwidth]{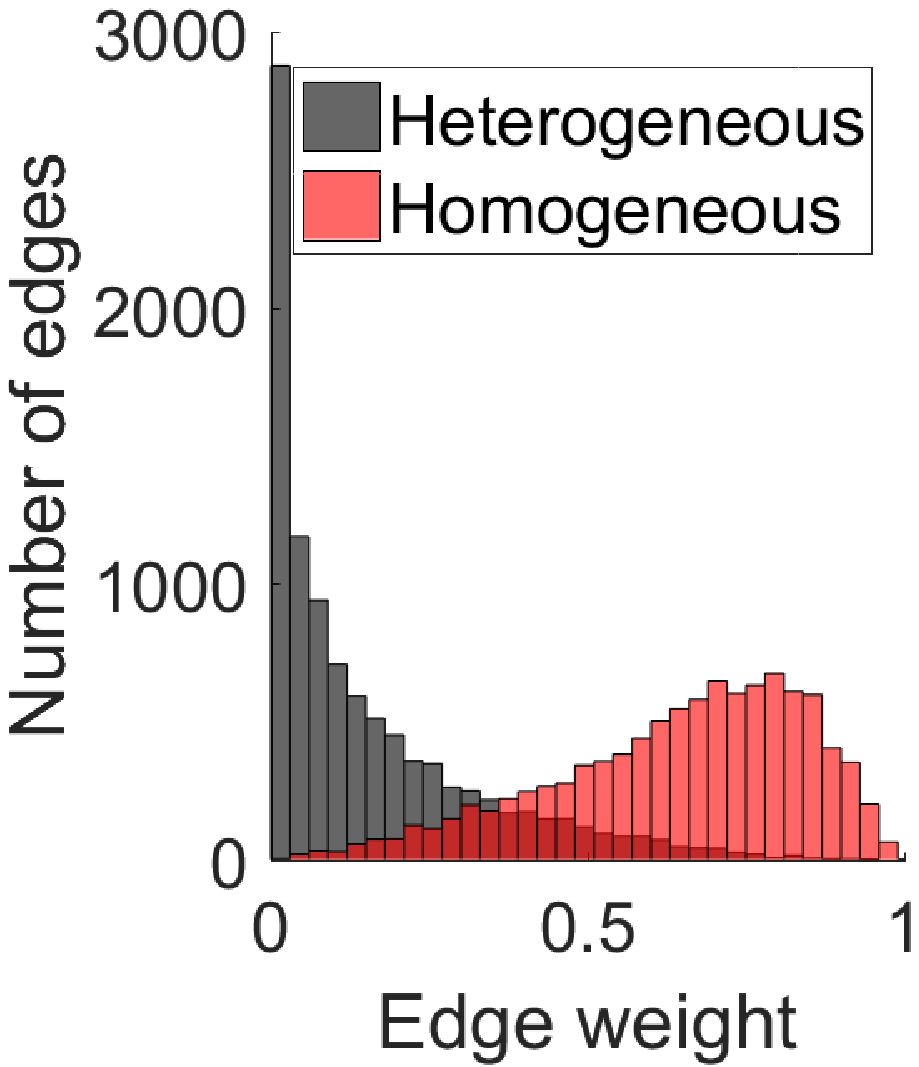}
		\caption{} \label{icmlwe}
	\end{subfigure}
	\hspace{-0.2in}
	\begin{subfigure}[c]{0.5\textwidth}
		\centering		\includegraphics[width=\columnwidth]{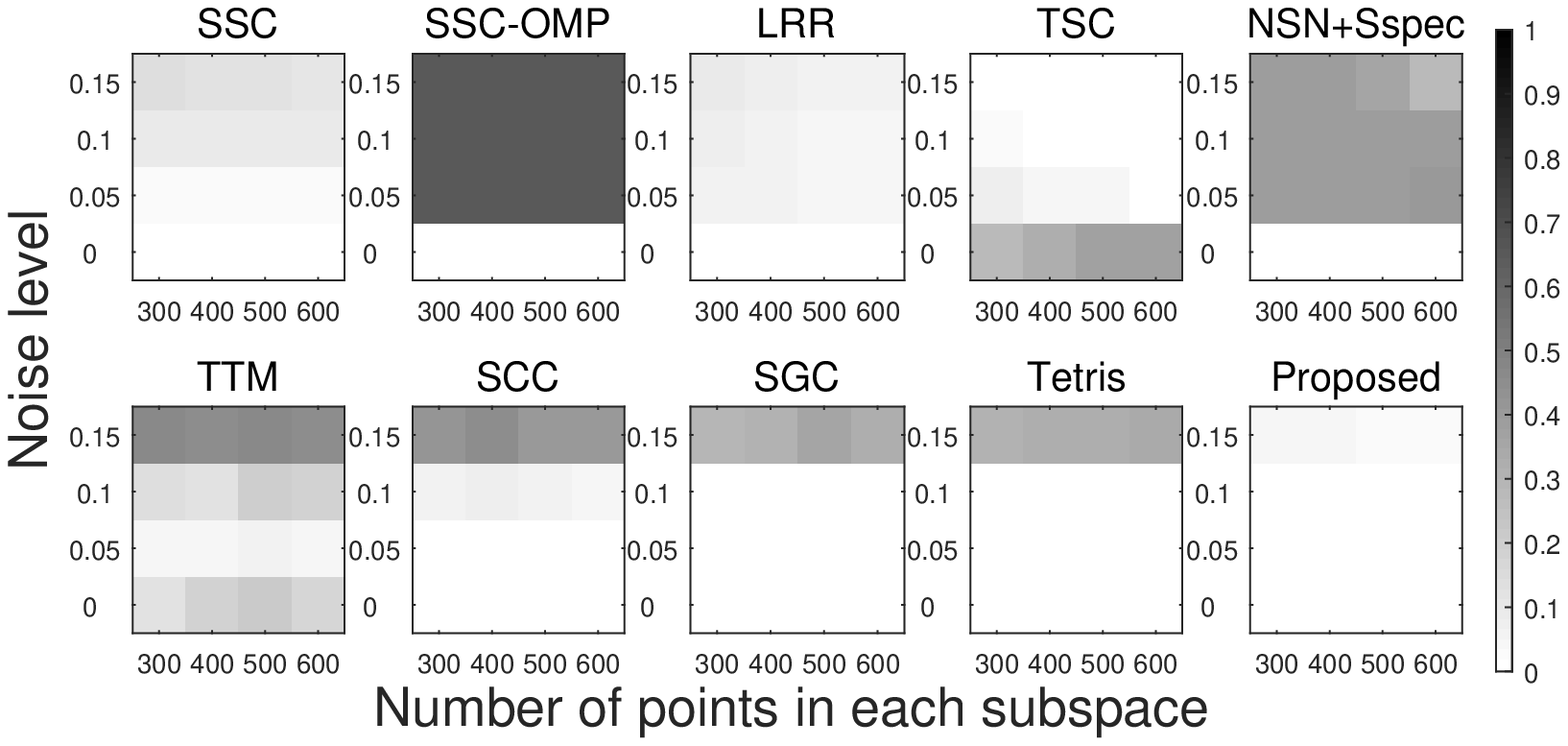}
		\caption{}\label{fig:error}
	\end{subfigure}
	\hspace{-0.2in}
	\begin{subfigure}[c]{0.28\textwidth}
		\centering		\includegraphics[width=\columnwidth]{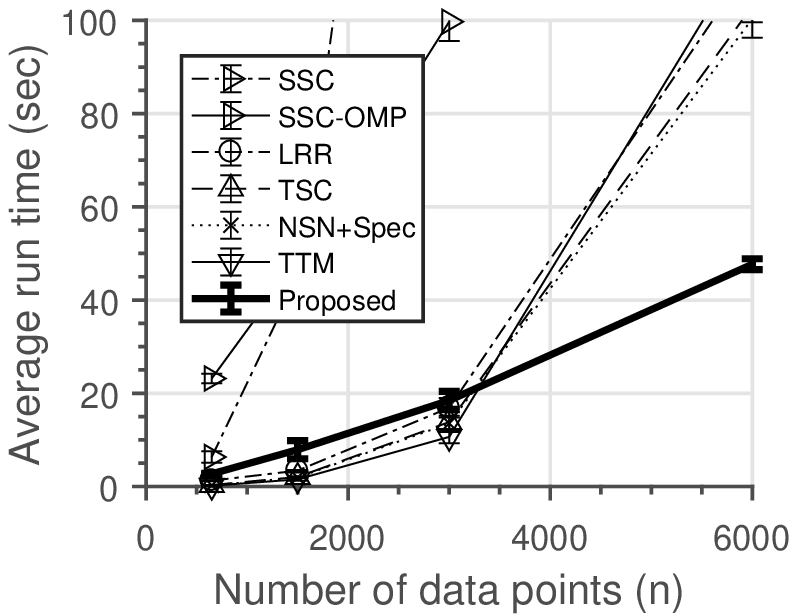}
		\caption{}\label{icmlspeed}
	\end{subfigure}
	\caption{\footnotesize{ (a) Distribution of edge weights for the Hopkins 155 data set. Notice that homogeneous edges have larger weights on average. This implies that the hypergraph constructed from tensor-based approaches respects our model. (b) Fractional error of various algorithms. We report the fractional error of each algorithm for varying $n/k$ and $\sigma$ \rev{}{(a lighter color implies a lower error fraction)}. Note that we include iterative algorithms (SCC, SGC, Tetris) although they cannot be utilized in the sketching scenario. We can see that our approach has a comparable performance to the state of the arts. (c) Average run time comparison with prior subspace clustering algorithms. We can observe that our proposed algorithm scales nearly linearly in $n$ while others do not.}}
	\vskip -0.25in
\end{figure*}

\subsection{Subspace clustering} 
\label{subsec:con} 
It is well known that hypergraph clustering is closely related to computer vision applications such as subspace clustering~\cite{agarwal2006higher}.
In  the subspace clustering problem, one is given  with  $n$ data points $\mathbf{x}_1,\mathbf{x}_2,\ldots, \mathbf{x}_n \in \mathbb{R}^{\ell}$ in a high dimensional ambient space.
 The $n$ data points are partitioned into $k$ groups, and data points in the same group approximately lie on the same subspace, each of dimension at most $m<\ell$. 
 The goal is to recover the hidden partition of the $n$ data points based on certain measurements. 
Among various approaches, tensor-based algorithms measure \emph{similarities between the data points} to recover the cluster~\cite{govindu2005tensor, agarwal2006higher, chen2009spectral}. 
 More specifically, they construct a weighted $d$-uniform ($d\geq m+2$) hypergraph $([n], \{W_e\}_{e\in \mathcal{E}})$, in which each edge weight represents the similarity of the corresponding $d$ points.
One typical approach to measure the similarity between $d$ data points is based on the hyperplane fitting. 
More specifically, denoting by $\fit(\cdot)$ the error of fitting an $m$-dimensional affine subspace to $d$ data points, one may set $W_e = \exp\left(-\fit(\mathbf{x}_{i_1},\ldots ,\mathbf{x}_{i_d}) \right)$ for $e =\{i_1,i_2,\cdots, i_d\}$.
Note that $W_e \simeq 1$ if the $d$ data points are approximately on the same subspace, and $W_e \simeq 0$ if the data points cannot be fit on a single subspace.

Consider a set of $d$ data points of the same cluster, which approximately lie on the same subspace by definition.
The edge weight corresponding to these $d$ data points will be approximately $1$, and one may model the edge weight as a random value whose expected value is close to $1$. 
Similarly, one may model the edge weights of heterogeneous edges by a random variable whose expected value is close to $0$.\footnote{Indeed, the edge weight of a heterogeneous edge can be very close to $1$, i.e., the fitting error can be close to $0$. 
	This may happen when $d$ data points, which are from different subspaces, are well aligned with another single subspace.
	Such a coincidence, however, happens with very low probability, and hence we simply treat these atypical events as statistical noise.} 

Clearly, our weighted SBM can precisely capture the above hypergraph model since our model only assumes that the average weights of homogeneous edges are larger than those of heterogeneous edges. 
We verify this claim using a real data set.  
Hopkins 155 is the most widely used dataset for the subspace clustering problem~\cite{tron2007benchmark}.
We first set $d=8$ and $m=3$, and then randomly sample $10000$ homogeneous edges and $10000$ heterogeneous edges. 
The  empirical  distributions of edge weights are shown in Fig.~\ref{icmlwe}. 
We can see that the homogeneous edges have larger weights on average than the heterogeneous edges, well respecting the weighted SBM.

\subsection{Sketching algorithms for subspace clustering}
Modern subspace clustering algorithms involve a large number of data points lying on a high-dimensional space, i.e., $n$ and $\ell$ are very large.
Hence, storing the entire raw data points is prohibitive, and one may have to resort to the \emph{sketch} of the data set. 
A sketch can be viewed as a summary of the dataset, containing sufficient information of the data set.

\rev{}{As evidenced by the preceding section, we assume that the weighted hypergraph constructed from the data points follows the model in Sec.~\ref{sec:model}. Under this assumption, subspace clustering can be done by clustering nodes of the weighted hypergraph.} The following corollary asserts that one can exactly solve the subspace clustering problem with a sketch consisting of the weights of randomly chosen hyperedges.\footnote{\rev{}{We note that one may carefully choose similarity entries in order to achieve a more informative sketch than our random one, at the cost of increased computational complexity for sketch construction.}}
We now state a corollary, a consequence of Thm.~\ref{thm:main} and \ref{thm:main2}. 
\begin{corollary} \label{coro:sketch}
	Suppose that $n_{max}/n_{min} =O(1)$ and $\alpha_n=1$. Then, HSC is weakly consistent if $\binom{n}{d}s_n=\omega(n)$, and HSCLR is strongly consistent if $\binom{n}{d}s_n\geq c_8 \cdot n\log n$ for some constant $c_8>0$. Moreover, the computational complexities of HSC, HSCLR reduce to $\max\{\binom{n}{d}s_n, n(\log n)^k\}$.
\end{corollary}

\rev{}{\begin{remark}
One can sketch data more aggressively if the subspaces are not similar to each other~\cite{heckel2017dimensionality}.
This is also captured in Corollary~\ref{coro:sketch} as follows.
As a concrete example, consider two subspaces of dimension $m$ and a heterogeneous edge $e$.
When the two subspaces are moving farther away from each other, the fitting error of $e$ increases. Thus, $W_e$ approaches $0$, and hence $p-q$ increases. 
Since the sample complexity is inversely proportional to $p-q$,\footnote{ The sufficient condition in Thm.~\ref{thm:main2} reads $\frac{(p-q)^2}{p}\binom{n}{d}s_n\geq C$ for some quantity $C$.} one can sketch more aggressively.
\end{remark}}

Corollary~\ref{coro:sketch} implies that our sketching method can reduce the storage overhead from $O(n\ell)$ to $O(n\log n)$.
We now evaluate our sketching algorithm.
The relevant parameters are $n$, $k$, $\ell$, $m$, $d$, and $s_n$: in an ambient dimension of $\ell=50$, we randomly generate $k$ subspaces each being of dimension of $m=3$; for each subspace, we randomly sample $n/k$ points and perturb every point with  Gaussian noise of variance $\sigma^2$; we set edge size $d=5$ and sampling probability $s_n$.
We first implement HSCLR in MATLAB\footnote{We observe a large constant in the computational complexity of the geometric $k$-clustering, and hence we implement HSCLR with an efficient $k$-means algorithms for the experiments.}.
We then compare HSCLR with other prior algorithm\footnote{Sparse subspace clustering (SSC)~\cite{elhamifar2013sparse}, a variant of SSC using OMP (SSC-OMP)~\cite{dyer2013greedy}, subspace clustering using low-rank representation (LRR)~\cite{liu2013robust}, thresholding-based subspace clustering (TSC)~\cite{heckel2015robust}, subspace clustering using nearest neighborhood search (NSN+Spec)~\cite{park2014greedy}, and tensor trace maximization (TTM)~\cite{JMLR:v18:16-100}. Note that SCC~\cite{chen2009spectral}, SGC~\cite{jain2013efficient}, and Tetris~\cite{JMLR:v18:16-100}) are not applicable to the sketching scenario due to their iterative natures.}, adopting the experimental setups from~\cite{park2014greedy} and \cite{JMLR:v18:16-100}.

We first measures the performance of various algorithms.
We set $k=3$ and $\binom{n}{d}s_n =5k^{d-1} n\log n /d$, and report the average fractional errors of each algorithm over $20$ trials for $(n/k, \sigma) \in \{300,400,500,600\}\times \{0, 0.05, 0.1, 0.15\}$ in Fig.~\ref{fig:error}.
Observe that our algorithm matches the state-of-the-art performance.
We also measures the run time of the algorithms.
We set $k=2$, $\sigma = 0.025$, $n\in \{750, 1500, 3000, 6000\}$,  $\binom{n}{d}s_n$ $=5k^{d-1} n\log n /d$, and report the average run time over $10$ trials.  
Fig.~\ref{icmlspeed} shows that the runtime of our proposed algorithm scales nearly linearly in $n$.

\rev{}{\subsection{Other applications}
Apart from subspace clustering, there are many applications in which $d$-wise similarities can carry more information than pairwise ones.
Those include other computer vision applications (such as geometric grouping~\cite{chen2009spectral,govindu2005tensor} and high-order matching~\cite{duchenne2011tensor}), tagged social networks~\cite{ghoshal2009random}, biological networks~\cite{michoel2012alignment} and co-authorship networks~\cite{yang2015defining}.
We remark that while our model assumes equal-sized hyperedges, the HSCLR algorithm is applicable even when the size of hyperedges vary, which is the case for some of these applications.
However, the success of the refinement step is contingent upon whether or not the average weight of homogeneous edges is larger than that of heterogeneous edges.
While this assumption is shown to hold for the subspace clustering problem, whether or not this assumption holds for the other applications is an interesting future direction.}

\section{Conclusion} \label{sec:concl}
	In this paper, we develop two hypergraph clustering algorithms: HSC and HSCLR.
	Our main contribution lies in performance analysis of them under a new hypergraph model, which we call the weighted SBM.
	Our results improve upon the state of the arts, and firstly settle the order-optimal results.
	Further, we show that HSCLR achieves the information-theoretic limits of a certain hypergraph model.
We also develop a sketching algorithm for subspace clustering based on HSCLR, and empirically show that the new algorithm outperforms the existing ones. 

We conclude our paper with future research directions.
	\begin{itemize}
		\item {\bf Detection threshold:} In~\cite{angelini}, a sharp threshold for detection is conjectured. Further, the non-backtracking method is conjectured to be optimal. Proving these conjectures still remains open. The optimality of HSC is also open. 
		\item {\bf Consistency threshold:} The fundamental limits for weak/strong consistency under the general weighted SBM are unknown. 
An important open problem is to characterize the general limits in terms of the model parameters $(n,d,k,p,q,\alpha_n)$.
	\end{itemize}
	
\appendices


\section{}\label{pf:spec}
We first note that the overall structure of the proof resembles the ones in~\cite{friedman1989second,feige2005spectral}, except that the entries of A are not independent.  
This is because each hyperedge's weight is added to more than one entries of $\Az$ in our case, resulting in dependency structure between all elements of the matrix. 
See Remark~\ref{rmk2} for more details.

We begin with some preliminaries: Let $\nu := \binom{n-2}{d-2}p\alpha_n \geq  \max_{i,j} \ex [\mathbf{A}_{i,j}]$; let $B:=\{ \mathbf{x} \in \mathbb{R}^n ~:~ \|\mathbf{x}\|_2 \leq 1 \}$; let $D_{\delta} := \left\{ \mathbf{x} = (x_1,x_2,\ldots, x_n) \in B~:~ \frac{\sqrt{n} x_i}{\delta} \in \mathbb{Z} \right\}$; for a matrix $\mathbf{C}$, $\den_{\mathbf{C}} (\mathcal{A},\mathcal{B}) := \sum_{i\in \mathcal{A}} \sum_{j \in \mathcal{B}} \mathbf{C}_{i,j}$; and for a matrix $\mathbf{C}$ and a subset $I$, let $\mathbf{C}^I$ be the matrix obtained from $\mathbf{C}$ by zeroing out all rows and columns in subset $I$.
The following large deviation results will be frequently used throughout the proof:
\begin{lemma} \label{machinery}
	Let $S = \sum_{i=1}^n X_i$, where $0\leq X_i \leq b$ for each $i$ for $b> 0$. There exist constants $\cseven>0$ depending only on $b$ such that the following holds for any $a\geq \ex [S]$ and $k\geq\cseven$:
	\[
	\Pr(S>k\cdot a) \leq \exp\big(-\frac{1}{2b} k \log k \cdot
	a\big).
	\]
	
\end{lemma}

\begin{IEEEproof}
	See Appendix~\ref{pf:largedev}.
\end{IEEEproof}

We consider the most challenging case where $\binom{n}{d}\alpha_n =\Theta(n)$, i.e., $\nu =\Theta(1/n)$.
First, note that $\mathbf{P} - \ex [\mathbf{A}]$ is a diagonal matrix whose entries are $O(\nu)$. Hence, $\|\mathbf{P} - \ex [\mathbf{A}]\|=O(\sqrt{n\nu})$.
Thus, it suffices to show that \begin{align}
\left\|\Az - \ex [\mathbf{A}]\right\| = O(\sqrt{n\nu})\,. \label{WTS}
\end{align}

\begin{lemma}\label{lem:sup}
	Let $\mathbf{C}$ be a $n\times n$ matrix. For $0<\delta <1$,
	\begin{align*}
	\|\mathbf{C}\| \leq (1-3\delta )^{-1} \max_{\mathbf{x}\in D_{\delta}} \left|\mathbf{x}^T \mathbf{C} \mathbf{x}\right|.
	\end{align*}
\end{lemma}
\begin{IEEEproof}
	See Appendix~\ref{pf:sup}.
\end{IEEEproof}

Due to Lemma~\ref{lem:sup}, one can replace \eqref{WTS} with a more tractable statement at the cost of the constant: $\sup_{\mathbf{x} \in D_{\delta}} \left| \mathbf{x}^T \left(\Az - \ex [\mathbf{A}] \right) \mathbf{x}\right| = O(\sqrt{n\nu})$.
For a vector $ \mathbf{x} = (x_1,x_2, \ldots, x_n) \in D_{\delta}$, define $S_{\delta}(\mathbf{x}) := \left\{(i,j) ~:~ |x_ix_j|<\delta^2\sqrt{\frac{\nu}{n}} \right\}$ for $0<\delta<1$.
Then, one has:
\begin{align*}
&\sup_{\mathbf{x} \in B} \big| \mathbf{x}^T \big(\Az - \ex [\mathbf{A}] \big) \mathbf{x}\big|= \sup_{\mathbf{x} \in B} \bigg| \sum_{(i,j)}\left[  \mathbf{A}^0 _{i,j} x_ix_j \right] - \mathbf{x}^T \ex [\mathbf{A}] \mathbf{x}\bigg| 
\end{align*}
Let $(T1) = \sup_{\mathbf{x} \in B} \bigg| \sum_{(i,j)\in S_{\delta}(\mathbf{x})} \big[\mathbf{A}^0 _{i,j} x_ix_j \big] - \mathbf{x}^T \ex [\mathbf{A}] \mathbf{x}\bigg|$ and $(T2) = \sup_{\mathbf{x} \in B} \bigg| \sum_{(i,j)\in S_{\delta}(\mathbf{x})^c} \big[ \mathbf{A}^0_{i,j} x_ix_j\big]\bigg|$. Then, the above quantity is bounded above by $(T_1) + (T_2)$.
We now show that each of $(T_1)$ and $(T_2)$ is $O(\sqrt{n\nu})$. 

\subsection{Proof of $(T1)$} We denote by $J$ the random subset of $[n]$ that corresponds to the removed rows and columns during the processing step (see step $2$ of Alg.~\ref{alg1}). For a sufficiently large constant  $\ctwelve>0$ (to be chosen later) and $I\subset[n]$, define the event
\[
E_I=\bigg\{ \sup_{\mathbf{x} \in D_{\delta}} \bigg| \sum_{(i,j)\in S_{\delta}(\mathbf{x})} \big[ \mathbf{A}^{I}_{i,j} x_ix_j\big]  - \mathbf{x}^T \ex[ \mathbf{A}] \mathbf{x}\bigg| > \ctwelve \cdot \sqrt{n\nu} \bigg\}.
\] 
Then, it is sufficient to show that $\Pr(E_J) \rightarrow 0$. 
Note that the following upper bound holds for:
\begin{align}
\!\!\!\!\!\!\Pr(E_J) &= \sum_{I\subset [n]}\left[\Pr(E_J, J = I) \right]\nonumber\\
&\leq  \sum_{|I| \leq (n\nu)^{-3} n}\left[\Pr(E_I, J = I) \right]+ \sum_{|I| \geq (n\nu)^{-3} n}\left[\Pr(E_I, J = I)\right]\nonumber\\
&\leq \sum_{|I| \leq (n\nu)^{-3} n}\left[\Pr(E_I)\right] + \sum_{|I| \geq (n\nu)^{-3} n}\Pr( J = I) \nonumber\\
&= \sum_{|I| \leq (n\nu)^{-3} n}\left[\Pr(E_I)\right] + \Pr\left( |J|\geq (n\nu)^{-3} n\right)\,. \nonumber
\end{align}
The following lemma bounds the number of removed rows (and columns).
\begin{lemma} \label{nolargedeg}
	For some $c_{thr}>0$ (depending only on $d$), there exists a constant $\ceight>0$ such that if $n\nu \geq \ceight$, then w.p. $1-\exp(-\Omega(n))$, $\left| \left\{ i ~:~ \den_{\mathbf{A}}(i,[n]) \geq c_{thr} \cdot n\nu \right \} \right| \leq (n\nu)^{-3}n$.
\end{lemma}
\begin{IEEEproof}
	See Appendix~\ref{pf:nolargedeg}.
\end{IEEEproof}

By Lemma~\ref{nolargedeg}, for $n\nu \geq \ceight$, $\Pr\left( |J|\geq (n\nu)^{-3} n\right) \leq  e^{-\Omega(n)}$.

As there are at most $2^n = e^{n\log 2 }$ many subsets of $[n]$, due to the union bound, the proof for $(T1)$ will be completed after showing that for a fixed $|I|\leq (n\nu)^{-3} n $, $\Pr(E_I) \leq O\left(e^{-2\log 2 n}\right)$.
Observe that
\begin{align*}
&\bigg|\sum_{(i,j)\in S_{\delta}(\mathbf{x})} \left[ \mathbf{A}^{I}_{i,j} x_i x_j\right]  - \mathbf{x}^T \ex [\mathbf{A}] \mathbf{x}\bigg| \\
&\leq  \underbrace{\bigg|\mathbf{x}^T \left(  \ex [\mathbf{A}^I]   -  \ex [\mathbf{A}] \right)\mathbf{x}\bigg|}_{(E1)}+\underbrace{\bigg|\sum_{(i,j)\in S_{\delta}(\mathbf{x})^c}\left[ \ex [\mathbf{A}^{I}_{i,j}]x_ix_j\right] \bigg|}_{(E2)} \\
&+ \underbrace{\bigg|\sum_{(i,j)\in S_{\delta}(\mathbf{x})} \left[ \left(\mathbf{A}^{I} _{i,j} -\ex [\mathbf{A}^{I}_{i,j}]\right) x_ix_j\bigg]\right|}_{(E3)}\,,
\end{align*}
and hence we will show that there exist $\cthirteen, \cfourteen,\cfifteen>0$ such that $\sup_{\mathbf{x} \in D_{\delta}} (E1) \leq \cthirteen \sqrt{n\nu}$, $\sup_{\mathbf{x} \in D_{\delta}} (E2) \leq \cfourteen \sqrt{n\nu}$, and $\sup_{\mathbf{x} \in D_{\delta}} (E3) \leq \cfifteen \sqrt{n\nu}$ with probability $1-O(e^{-2n\log2 })$, respectively.
Having shown these, the proof for $(T1)$ is completed by taking $\ctwelve := \cthirteen + \cfourteen + \cfifteen$.

\begin{itemize}
	\item[(i)]$(E1)$:  As $|I|\leq (n\nu)^{-3} n$, 
	\begin{align*}
	&\left|\mathbf{x}^T \left(  \ex [\mathbf{A}^I]   -  \ex [\mathbf{A}] \right)\mathbf{x}\right| \leq \left\|\ex  [\mathbf{A}^I] -\ex [\mathbf{A}] \right\| \\
	&\leq \left\|\ex  [\mathbf{A}^I] -\ex [\mathbf{A}]\right \|_F \leq \sqrt{2(n\nu)^{-3} n^2 \cdot  \nu^2 } = \sqrt{2} (n\nu)^{-1/2}.
	\end{align*}
	Hence, by taking $\cthirteen = \sqrt{2}$, $\sup_{\mathbf{x} \in D_{\delta}} (E1) \leq \cthirteen \sqrt{n\nu}$ holds with probability $1$ for $n\nu \geq 1$.
	\item[(ii)]$(E2)$: As $\nu\geq \max_{i,j}\ex\left[\mathbf{A}_{i,j}\right]$,
	\begin{align*}
	&\bigg|\sum_{(i,j)\in S_{\delta}^c} \left[\ex [\mathbf{A}^{I}_{i,j}]x_ix_j\right] \bigg|
	\leq \nu \sum_{\substack{(i,j)\in S_{\delta}(\mathbf{x})^c \\ i\neq j }}  \left|x_ix_j\right| \\
	&= \nu \sum_{\substack{(i,j)\in S_{\delta}(\mathbf{x})^c \\ i\neq j }} \frac{x_i^2 x_j^2}{\left| x_ix_j\right|} \overset{(a)}{\leq} \frac{1}{\delta^2} \sqrt{n\nu} \sum_{\substack{(i,j)\in S_{\delta}(\mathbf{x})^c \\ i\neq j }} x_i^2x_j^2\overset{(b)}{\leq} \frac{1}{\delta^2}\sqrt{n\nu}\,,
	\end{align*}
	where ($a$) is due to the definition of $S_{\delta}(\mathbf{x})$, and ($b$) follows since $\left\| \mathbf{x}\right\|\leq 1$. Hence, by taking $\cfourteen = \frac{1}{\delta^2}$, $\sup_{\mathbf{x} \in D_{\delta}} (E2) \leq \cfourteen \sqrt{n\nu}$ holds with probability $1$.
	
	\item[(iii)]$(E3)$:
	Let $ \mathbf{x} = (x_1, x_2,\ldots , x_n) \in D_{\delta}$ be fixed. We have 
	\begin{align*}
	&\sum_{(i,j)\in S_{\delta}(\mathbf{x})}\left[ \left(\mathbf{A}^{I} _{i,j} -\ex  [ A^{I}_{i,j}]  \right) x_ix_j \right]\\
	&= \sum_{\substack{(i,j)
			\in S_{\delta}(\mathbf{x})  \\ i\neq j}} \bigg[ x_i x_j \mathbb{I}\{i\notin I, j\notin I\} \sum_{\substack{e\in \mathcal{E} \\ \{i,j\}\subset e}} \left[W_e -\ex[ W_e] \right] \bigg]\\
	&= \sum_{e \in \mathcal{E}} \underbrace{\bigg[ \left(W_e -\ex [W_e] \right)\sum_{\substack{(i,j)
				\in S_{\delta}(\mathbf{x})  \\ i\neq j,~ \{i,j\}\subset e}} \left[x_ix_j \mathbb{I}\{i\notin I, j\notin I\} \right]  \bigg] }_{=:Y_e}\,.
	\end{align*}
	Note that $\{Y_e\}_{e\in\mathcal{E}}$ is a collection of independent random variables.
	To apply Bernstein inequality to $\sum_{e\in \mathcal{E}}Y_e$, we do some preliminary calculations.
	First, it easily follows from the definition of $S_{\delta}$ that
	\begin{align*}
	&|Y_e| \leq \bigg| \left(W_e -\ex [W_e] \right)\sum_{\substack{(i,j)
			\in S_{\delta}(\mathbf{x})  \\ i\neq j,~ \{i,j\}\subset e}} \left[x_ix_j \mathbb{I}\{i\notin I, j\notin I\} \right]  \bigg| \\
	&\leq  \sum_{\substack{(i,j)
			\in S_{\delta}(\mathbf{x})  \\ i\neq j,~ \{i,j\}\subset e}} \left|x_ix_j  \right| \leq \delta^2\sqrt{\frac{\nu}{n}} \cdot 2\binom{d}{2} \leq d^2 \delta^2\sqrt{\frac{\nu}{n}}\,.
	\end{align*} 
	Next, we compute a bound on the sum of variances:
	\begin{align*}
	&\sum_{e\in \mathcal{E}} \ex [Y_e^2]\\
	&\overset{(a)}{\leq}    \sum_{e\in \mathcal{E} } \bigg[ d^2 \ex [W_e^2] \sum_{\substack{(i,j)
			\in  S_{\delta}(\mathbf{x})  \\ i\neq j,~ \{i,j\}\subset e}}\left[ x^2_ix^2_j \mathbb{I}\{i\notin I, j\notin I\}\right] \bigg] \\
	&\overset{(b)}{\leq} d^2 \sum_{e\in \mathcal{E}} \bigg[\ex [W_e] \sum_{\substack{(i,j) \\ i\neq j, \{i,j\}\subset e}} \left[x^2_i x^2_j\right] \bigg]\\
	&=d^2 \sum_{\substack{(i,j) \\ i\neq j}} \bigg[x^2_i x^2_j \ex[\mathbf{A}_{i,j}]\bigg] \leq d^2 \nu \sum_{\substack{(i,j) \\ i\neq j}} x^2_i x^2_j    \overset{(c)}{\leq} d^2 \nu \,,
	\end{align*}
	where ($a$) is due to $(\sum_{i=1}^k a_i)^2 \leq k \sum_{i=1}^k a_i^2$; ($b$) follows since $W_e \in [0,1]$; ($c$) follows since $\left\| \mathbf{x}\right\|\leq 1$.

	Thus, Bernstein inequality yields:
$\Pr\big(  \big|\sum_{e\in \mathcal{E}} Y_e \big| \geq t \big) \leq 2\exp\big(-\frac{t^2/2}{d^2\nu + \frac{1}{3} d^2\delta^2\sqrt{\frac{\nu}{n}}t}\big)$, we have
	\begin{align*}
	&\Pr\bigg(  \bigg|\sum_{e\in E} Y_e \bigg| \geq \cfifteen \sqrt{n\nu} \bigg)  \\
	&\leq 2\exp\left(-\frac{\cfifteen^2  }{2d^2 + \frac{2}{3} d^2\delta^2\cfifteen } n\right)\,.
	\end{align*}
	
	As $|D_{\delta}| = e^{\Theta(n)}$, the union bound yields
	\begin{align*}
	&\Pr\bigg( \sup_{\mathbf{x} \in D_{\delta}} \bigg|\sum_{e\in E} Y_e \bigg| \geq \cfifteen \sqrt{n\nu}  \bigg) \\
	&\leq e^{\Theta(n)}\Pr\bigg( \bigg|\sum_{e\in E} Y_e \bigg| \geq \cfifteen \sqrt{n\nu}  \bigg) \\
	&\leq  2e^{\Theta(n)}\exp\bigg(-\frac{\cfifteen^2  }{2d^2 + \frac{2}{3} d^2\delta^2\cfifteen } n\bigg)\,,
	\end{align*}
	and hence by choosing $\cfifteen$ sufficiently large, one can ensure that $\sup_{\mathbf{x} \in D_{\delta}} (E3) \leq \cfifteen \sqrt{n\nu}$ w.p. $1-O(e^{-2n\log2 })$.
	
\end{itemize}

Since $n\nu \geq \ceight$ and 
$n\nu \geq 1$, $n\nu$ should be greater than or equal to $\max\{\ceight,~1\}$, so one can take $\ceee = \max\{\ceight,~1\}$.

\subsection{Proof of ($T2$)} 
This case immediately follows from a celebrated combinatorial technique proposed in \cite{friedman1989second}.
We summarize their results. 
\begin{definition} \label{def:bdd}
	We say \emph{the bounded density property} holds with constants $\alpha,\beta,\gamma>0$ if the following two hold:
	\begin{enumerate}
		\item For each node $u$, $\den_{\Az}(u, [n]) \leq \alpha \cdot n \nu$.
		\item
		For any two subsets $A,B$,
		either $\den_{\Az}(\mathcal{A},\mathcal{B}) \leq \beta \cdot \nu |\mathcal{A}||\mathcal{B}|$ 
		or 		
		$\den_{\Az}(\mathcal{A},\mathcal{B}) \log \frac{\den_{\mathbf{A}_0}(\mathcal{A},\mathcal{B})}{\nu |\mathcal{A}||\mathcal{B}|} \leq \gamma \cdot \max\{|\mathcal{A}|,|\mathcal{B}|\}\log \frac{n}{\max\{|\mathcal{A}|,|\mathcal{B}|\}}$.
	\end{enumerate}
\end{definition}

\begin{proposition}[\!\!\cite{friedman1989second,feige2005spectral}] \label{T2}
	If the bounded density property holds with some constant $\alpha, \beta, \gamma$, then $(T2)=O(\sqrt{n\nu})$.
	
\end{proposition}

Therefore, one only needs to show that the bounded density property holds with high probability to finish the proof.

\begin{lemma} \label{lem:bdd}
	With probability $1-O(n^{-1})$, the bounded density property holds with some constants $\cnine,\cten,\celeven$.
\end{lemma}
\begin{IEEEproof}
	See Appendix~\ref{pf:bdd}.
\end{IEEEproof}

\section{}\label{pf:info}
For notational simplicity, as $k=2$, we represent partition functions $\phsc,\Psi$ by binary vectors $\mathbf{X},\mathbf{Z} \in \{0,1\}^n$.
We define some notations: Let $\ww = \left[W_e\right]_{e\in \mathcal{E} }$; for a vector $\mathbf{V}=\left[\mathbf{V}_i\right]_{1\leq i\leq n} \in \{0,1\}^{n}$ and $e = \{i_1,i_2,\ldots, i_d \}\in \binom{[n]}{d}$, let $\ye ({\bf V}) = \mathbb{I}\{V_{i_1} = V_{i_2}=\ldots V_{i_d}\} $; let $\Ye(\mathbf{V}) =\left[\ye(\mathbf{V})\right]_{e\in \mathcal{E}}$.

A straightforward calculation yields for any two binary vectors $\mathbf{X}$ and $\mathbf{Y}$, the likelihood of $\mathbf{X}$ is greater than that of $\mathbf{Y}$ if and only if $\dd(\ww ,\Ye(\mathbf{X}) ) < \dd(\ww ,\Ye(\mathbf{Y}) )$, 
where $\dd(\mathbf{X}, \mathbf{Y}) := \left|\left\{i\in[n]~:~ \mathbf{X}_i \neq \mathbf{Y}_i  \right\} \right|$ for any $\mathbf{X}$ and $\mathbf{Y}$. 

To make HSCLR better adapted to the model, we modify the algorithm as follows:
\begin{enumerate}
	\item We apply HSC to $([n],\mathbf{W}')$, where $\mathbf{W}'$ is obtained from $\mathbf{W}$ by replacing the erasure weights $\era$'s with $0$'s.
	\item  We then employ a likelihood-based refinement rule:
	\begin{align*}
	&\mathbf{X}_i \leftarrow \begin{cases}
		\mathbf{X}_i &\text{if}~ \dd(\ww,\Ye(\mathbf{X})) < \dd(\ww,\Ye(\mathbf{X}\oplus \mathbf{e}_i));\\
		\mathbf{X}_i\oplus 1 &\text{otherwise.}
		\end{cases}
	\end{align*} 
\end{enumerate}
\begin{remark}
	Notice that one can employ such a likelihood-based estimator only when edge distributions are fully specified.
\end{remark}

We now begin the main proof. 
We consider the most challenging regime where $\binom{n}{d}p = \Theta(n \log n)$, and suppose 
\begin{align}
\binom{n}{d} \alpha_n	 \geq (1+\epsilon)\frac{2^{d-2}}{d} \frac{n \log n}{(\sqrt{1-\theta}-\sqrt{\theta})^2} \label{cond}
\end{align} for a fixed $\epsilon>0$.
For simplicity, we assume that $n$ is even, and  fix the ground truth to be  $\mathbf{A} = ( \underbrace{1,\ldots, 1}_{n/2}, \underbrace{0,\ldots,0}_{n/2} )$;
for other  cases, the proof follows similarly.

Let $\mathbf{X}$ be the output of the first stage. 
By Thm.~\ref{thm:main}, one can see that $\mathbf{X}$ is weakly consistent.
Without loss of generality, we assume  for an arbitrarily small $\ee>0$ that
\begin{align*}
\mathbf{X} = (\underbrace{0,\ldots,0}_{\ee n}, \underbrace{1,1,1,\ldots,1,1, 1}_{n/2-\ee n}, \underbrace{1,\ldots,1}_{\ee n}, \underbrace{0,0,0,\ldots,0,0, 0}_{n/2-\ee n} )\,.
\end{align*}
Indeed, $\mathbf{X}$ needs not have the same number of $0$'s and $1$'s but the other cases can be handled similarly using the same arguments.

As in the proof of Thm.~\ref{thm:main2} (see Sec.~\ref{pf:thm2}), due to the union bound, it is enough to show that the probability of having node $1$'s affiliation incorrect after refinement is $o(n^{-1})$, i.e.,
\begin{align*}
\Pr\left(\text{node $1$ is incorrect after refinement}\right) = o(n^{-1})\,.
\end{align*}
By the new refinement rule,
\begin{align*}
&\Pr\left(\text{node $1$ is incorrect after refinement}\right) \\
&= \Pr\bigg( 0 < \dd(\ww,\Ye(\mathbf{X}\oplus \mathbf{e}_1)) -\dd(\ww,\Ye(\mathbf{X})) \bigg) \,.
\end{align*}
The following lemma states that the difference of hamming distances can be viewed as the sum of random variables.
\begin{lemma} \label{lem:comp}
$P_i,P_i' \overset{\text{i.i.d.}}{\sim} \bern(\alpha_n) $ and  $\Theta_i,\Theta_i' \overset{\text{i.i.d.}}{\sim} \bern(\theta)$. Then,
	\begin{align*}
	&\dd(\ww,\Ye(\mathbf{X}\oplus \mathbf{e}_1)) -\dd(\ww,\Ye(\mathbf{X}))\\
	&= \sum_{i=1}^{2 \binom{n/2-\ee n}{d-1}}  P_i(2\Theta_i-1 )) +\sum_{i=1}^{2\binom{n/2}{d-1}-2\binom{n/2-\ee n}{d-1}} P_i' \left(1-2\Theta_i'\right)\,.
	\end{align*}
\end{lemma}
\begin{IEEEproof}
	See Appendix~\ref{pf:comp}.
\end{IEEEproof}
Let $V_1=\binom{n/2}{d-1}$ and $V_2=\binom{n/2-\ee n}{d-1}$. By Lemma~\ref{lem:comp}, 
\begin{align}
&\Pr\bigg( 0 <  \dd(\ww,\Ye(\mathbf{X}\oplus \mathbf{e}_1)) -\dd(\ww,\Ye(\mathbf{X})) \bigg) \nn\\
&=\Pr\left( -\sum_{i=1}^{2V_1 - 2V_2} P_i' \left(1-2\Theta_i'\right) < \sum_{i=1}^{2V_2}  P_i(2\Theta_i-1 )) \right) \nn\\
&\overset{(a)}{\leq} \Pr\left( -\sum_{i=1}^{2V_1-2V_2} P_i'  < \sum_{i=1}^{2V_2}  P_i(2\Theta_i-1 )) \right) \nn\\
&\overset{(b)}{=}\Pr\left( -\sum_{i=1}^{O(\ee)V_1} P_i' < \sum_{i=1}^{2V_2}  P_i(2\Theta_i-1 )) \right)  \,. \label{last:333}
\end{align}
where ($a$) is due to $|1-2\Theta'_i|\leq 1$; ($b$) is due to $2V_1 - 2V_2 = O(\ee)V_1$.

In view of Lemma~\ref{marginal}, one can similarly show that 
\begin{align}
\Pr\left( \sum_{i=1}^{O(\ee)V_1} P_i' > \frac{V_1\alpha_n}{\sqrt{\log(1/\ee)}} \right) =o(n^{-1})\,,
\end{align}
provided that $\ee$ is sufficiently small.
Thus,
\begin{align*}
\eqref{last:333} \leq \Pr\left( -\frac{V_1\alpha_n}{\sqrt{\log(1/\ee)}}   < \sum_{i=1}^{2 V_2}  P_i(2\Theta_i-1 )) \right) +o(n^{-1})
\end{align*}

\begin{lemma} \label{lem:3}
	For an integer $K>0$, let $\{P_i\}_{i=1}^K\overset{\text{i.i.d.}}{\sim}{\sf Bern}(\alpha_n)$ and $\{\Theta_i\}_{i=1}^K\overset{\text{i.i.d.}}{\sim}{\sf Bern}(\theta)$.
	Then, for any $\ell>0$
	\begin{align*}
	&\Pr\left( \cmatrix   \sum_{i=1}^K   P_i(2\Theta_i-1) \geq -\ell \right) \\
	&\leq e^{\frac{1}{2}\ell - K\left(\alpha_n(\sqrt{1-\theta}-\sqrt{\theta})^2 +O(\alpha_n^2)\right)}\,.
	\end{align*} 
\end{lemma}

\begin{proof}
	See Appendix~\ref{pf:3}.
\end{proof}

By Lemma~\ref{lem:3}, 
\begin{align}
&\Pr\left( -\frac{V_1\alpha_n}{\sqrt{\log(1/\ee)}}   < \sum_{i=1}^{2 V_2}  P_i(2\Theta_i-1 )) \right) \nn \\
&\leq e^{\frac{1}{2}\frac{\cmatrix}{ \sqrt{\log(1/\ee)}}V_1\alpha_n - 2V_2\left(\alpha_n(\sqrt{1-\theta}-\sqrt{\theta})^2 +O(\alpha_n^2)\right) }\,. \label{last:3333}
\end{align}
Note that as $\alpha_n =o(1)$,
\begin{align*}
&\frac{1}{2}\frac{\cmatrix}{\sqrt{\log(1/\ee)}}V_1\alpha_n - 2V_2\left(\alpha_n(\sqrt{1-\theta}-\sqrt{\theta})^2 +O(\alpha_n^2)\right)\\
& = (1+o(1))\bigg[\frac{\cmatrix}{\sqrt{\log(1/\ee)}2^{d} }\frac{d}{n} \binom{n}{d}\alpha_n - 2\left(\frac{1}{2}-\ee\right)^{d-1} \nonumber\\
&~~~\cdot \frac{d}{n}\binom{n}{d}\left(\alpha_n(\sqrt{1-\theta}-\sqrt{\theta})^2 +O(\alpha_n^2)\right)\bigg]\\
& =(1+o(1))\bigg[\frac{\cmatrix}{\sqrt{\log(1/\ee)}2^{d} }\frac{d}{n} \binom{n}{d}\alpha_n - 2\left(\frac{1}{2}-\ee\right)^{d-1} \nonumber\\
&~~~\cdot\frac{d}{n}\binom{n}{d}\alpha_n(\sqrt{1-\theta}-\sqrt{\theta})^2\bigg]\\
&\to  -\frac{1}{2^{d-2}} \frac{d}{n}\binom{n}{d}\alpha_n(\sqrt{1-\theta}-\sqrt{\theta})^2
\end{align*}
as $\ee \to 0^+$ and $n\to \infty$.
Thus, $\eqref{last:3333} \leq e^{-(1+\epsilon/2)\log n} = o(n^{-1})$ for sufficiently large $n$ and small $\ee$.

\section{} \label{appenB}
To extend the analysis to the planted clique model, we need another type of error fraction, which is defined as follows:
\begin{align*}
\errr(\Phi):= \min_{\Pi\in \mathcal{P}} \max_{1\leq j\leq k} \frac{1}{n_j} |\{i\in \Psi^{-1}(j)~:~ \Pi(\Phi(i))\neq j  \}|\,.
\end{align*}
Note that $\errr$ characterizes the maximum value of within-cluster error fraction over all clusters.
Let us denote the smallest singular value of $\mathbf{B}$ (defined in Sec.~\ref{hsc}) by $\sigma$ and the size of smallest cluster by $n_{\min}$.
Then, following \cite{lei2015consistency}, one can prove the following result by tweaking the proof of Thm.~\ref{thm:main}:
\begin{theorem}\label{thm:mainprime}
For some $\cfourteen ,\cfifteen$, the following holds: if $\binom{n}{d}\alpha_n  \geq \cfourteen n$ and $\cfifteen(1+\epsilon) \frac{kn \binom{n-2}{d-2}}{\alpha_n n_{\min}^2 \sigma^2} < 1$, then w.p. exceeding $1-O(n^{-1})$,
	\begin{align}
	\errr(\phsc) \leq \cfifteen (1+\epsilon)\frac{kn \binom{n-2}{d-2}}{\alpha_n n_{\min}^2 \sigma^2}\,.
	\end{align}
\end{theorem}  

%

We now demonstrate how Thm.~\ref{thm:mainprime} guarantees the detection of planted clique when $s\geq c^*\cdot \sqrt{n}$ for some constant $c^*$.
To apply Thm.~\ref{thm:mainprime}, we need to first compute $\sigma$ of
\begin{align*}
\mathbf{B} = \left( \begin{array}{cc}
\frac{1}{2}\binom{s-2}{d-2} + \frac{1}{2}\binom{n-2}{d-2} & \frac{1}{2}\binom{n-2}{d-2} \\\frac{1}{2}\binom{n-2}{d-2}& \frac{1}{2}\binom{n-2}{d-2}
\end{array} \right).
\end{align*} 
Using the fact that the minimum singular value of $\left( \begin{array}{cc}
	a+b & a\\a& a
	\end{array} \right)$ is $\frac{2b}{\sqrt{4+(\frac{b}{a})^2}+2+\frac{b}{a}}$, we have
\begin{align*}
\sigma = \frac{2\binom{s-2}{d-2}}{\sqrt{4+\left(\frac{\binom{s-2}{d-2}}{\binom{n-2}{d-2}}\right)^2}+2+\frac{\binom{s-2}{d-2}}{\binom{n-2}{d-2}}} = \Theta \left(s^{d-2} \right)\,.
\end{align*} 
Hence, by Thm.~\ref{thm:mainprime} (as $\alpha_n=1$), 
Thus, whenever $s =\Omega(\sqrt{n})$,
\begin{align*}
\errr(\phsc) \leq   \cfifteen (1+\epsilon)\frac{2n \binom{n-2}{d-2}}{ s^2 \sigma^2} = O\left(\frac{ n^{d-1}}{ s^{2(d-1)})}\right) = O(1)\,.
\end{align*}

\section{} \label{sec:lem}
\subsection{Proofs of Lemma~\ref{largedev} and Lemma~\ref{machinery} } \label{pf:largedev}

Without loss of generality, we will prove the lemmas assuming that $\ex[X_i] > 0$ for all $i$.
We first obtain a useful bound on the moment generating function (mgf) of $S$. 
For an arbitrary $\lambda>0$,  
\begin{align}
&\ex[\exp\{\lambda S\}] = \ex  \left[\exp\left\{\lambda \left(\sum_{i=1}^n X_i\right)\right\}\right]\nn \\
&\leq \prod_{i=1}^n {\left( 1 + \frac{(e^{\lambda b}-1)}{b} \ex [X_i]\right)} \nn\\
&\leq \bigg(1+  \frac{\left(e^{\lambda b} -1\right)}{b}\frac{\sum_{i=1}^n \ex [X_i]}{n}   \bigg)^n, \label{mgfbdd}
\end{align}
where the first inequality holds since $\frac{e^{\lambda x}-1}{x} \leq \frac{e^{\lambda b}-1}{b}$ holds for all $0< x \leq b$, and the second inequality holds due to the AM-GM inequality.
We now prove the lemmas using this bound.

\subsubsection{Proof of Lemma~\ref{machinery}}
Using Markov's inequality and \eqref{mgfbdd},
\begin{align*}
&\Pr(S> x) = \Pr(e^{\lambda S}> e^{\lambda x}) \leq  \exp\{-\lambda x\} \ex [\exp\{ \lambda S\}] \nonumber \\
&\leq  \exp\{-\lambda x\}\left(1+  \frac{\left(e^{\lambda b} -1\right)}{b}\frac{\ex [S]}{n}   \right)^n.
\end{align*}
By choosing $\lambda = \frac{1}{b}\log\left(1+\frac{bx}{\ex [S]} \right)$, i.e., $x = \frac{(e^{b\lambda}-1)}{b}\ex[S]$, we have
\begin{align*}
&\Pr(S> x) \leq \exp\left\{-  \frac{x}{b}\log\left(1+\frac{bx}{\ex [S]}   \right)\right\} \left(1+\frac{x}{n}  \right)^n \\
&\leq \exp\left\{-  \frac{x}{b}\log\left(1+\frac{bx}{\ex [S]}   \right)\right\}\exp(x)\\
&= \exp\left[ - \frac{x}{b}  \cdot \left\{ \log\left(1+\frac{bx}{\ex[ S]}   \right) -b\right\} \right].
\end{align*}
By setting $x= k a$, we have
\begin{align*}
&\Pr(S > k\cdot a)  = \exp\left[ - \frac{ka}{b}  \cdot \left\{ \log\left(1+\frac{bka}{\ex[ S]}   \right) -b\right\} \right]\\
&\leq \exp\left[ - k\frac{\log\left(1+b k   \right) -b}{b}  \cdot a \right],
\end{align*}
where the inequality holds since $a\geq \ex [S]$.
Since $[\log(1+bk)-b] \sim \log(k)$, $\log(1+bk)-b \geq \frac{1}{2}\log(k)$ holds for all $k \geq c_7$, where $c_7$ is some positive constant depending only on $b$. 
Applying this inequality to the above bound completes the proof.

\subsubsection{Proof of Lemma~\ref{largedev}}
Since the proof bears great similarity to the conventional case~\cite{alon2004probabilistic}, we only show the upper bound.
Using Markov's inequality and \eqref{mgfbdd} with $b=1$, 
\begin{align*}
&\Pr\left(S  > (1+\delta) \ex[S] \right) = \Pr\left(e^{\lambda S} > e^{\lambda (1+\delta) \ex[S]} \right) \\
&\leq e^{-\lambda (1+\delta)\ex [S]} \left(1+ (e^\lambda -1)\frac{\ex [S]}{n}\right)^n
\end{align*}
By taking $\lambda = \log \left(1 + \delta \right)$, we obtain
\begin{align}
&\Pr\left(S  > (1+\delta) \ex[S] \right) \leq e^{- (1+\delta)\log \left(1 + \delta \right) \ex [S]} \left(1+\frac{\delta \ex[S]}{n}\right)^n \nn\\
&\leq e^{- (1+\delta)\log \left(1 + \delta \right)\ex [S] +\delta \ex[S]} \leq e^{- \frac{\delta^2}{2+\delta}\ex[S]}, \label{chernoff}
\end{align}
where the last equality holds since $\log (1+\delta ) \geq \frac{\delta}{1+\delta/2}$.
This completes the proof of the upper bound.

\subsection{Proof of Lemma~\ref{marginal}} \label{pf:marginal}
Assume that $|\nei (j)\cap \mathcal{E}_h| = c\cdot \gamma |\nei (j)|$ for some constant $c>0$. For simplicity, let us write $\sum_{e\in \nei (j)\cap \mathcal{E}_h} W_e $ as $\sum_{i=1}^{c\cdot \gamma |\nei (j)|} W_i$. Then we get:
\begin{align}
&\Pr\bigg( \sum_{e\in \nei (j)\cap \mathcal{E}_h} W_e > \frac{p\alpha_n |\nei (j)|}{\sqrt{\log(1/\gamma)}}  \bigg) \nn\\
&= \Pr\bigg( \sum_{i=1}^{c\cdot \gamma |\nei (j)|} W_i > \frac{1}{c\gamma\sqrt{\log(1/\gamma)}}\cdot  c\gamma p\alpha_n |\nei (j)|   \bigg)\,. \label{intt} 
\end{align}
From the proof of Lemma~\ref{largedev} (see \eqref{chernoff}), one can deduce the following:
\begin{corollary} \label{largedevpp}
	Let $S$ be the sum of $m$ mutually independent random variables taking values in $[0,1]$. For any $\delta>0$, we have
	\begin{align}
	\Pr\left( S > (1+\delta) \ex [S] \right)  \leq 	e^{-\{(1+\delta) \log (1+\delta) -\delta\} \ex[S] }\,. \label{upper:44}
	\end{align}
\end{corollary}
 We will apply Corollary~\ref{largedevpp} with $(1+\delta )= ({c\gamma\sqrt{\log(1/\gamma)}})^{-1} $. As $ ({c\gamma\sqrt{\log(1/\gamma)}})^{-1}\to \infty$ as $\gamma \to 0^+$, we may regard  $\delta$ to be an arbitrarily large constant.
Because $(1+\delta) \log (1+\delta) -\delta = (1+o(1))(1+\delta) \log (1+\delta) $ as $\delta \to \infty$, in what follows, we will replace the upper bound \eqref{upper:44} with $e^{-(1+\delta) \log (1+\delta)  \ex[S]}$:
\begin{align}
&\eqref{intt} \leq \left( \frac{1}{ec\gamma\sqrt{\log(1/\gamma)}}\right)^{-\frac{1}{c\gamma \sqrt{\log(1/\gamma)}}c\gamma p\alpha_n |\nei (j)|} \nn \\
&\leq \left( \frac{1}{ec\gamma\sqrt{\log(1/\gamma)}}\right)^{-\frac{1}{c\gamma \sqrt{\log(1/\gamma)}}c\gamma p\alpha_n |\nei (j)|}\,. \label{inttt}
\end{align}
Since we consider the regime $\binom{n}{d}\alpha_n = \Theta(n\log n)$, $p\alpha_n |\nei (j)| = c'\cdot \log n$ for some constant $c'>0$. Hence, the last term is equal to
\begin{align*}
&\bigg( \frac{1}{ec\gamma\sqrt{\log(1/\gamma)}}\bigg)^{-\frac{1}{ \sqrt{\log(1/\gamma)}}c'\cdot \log n}\\
&=\exp\bigg(-\log n \cdot\bigg\{ \sqrt{\log (1/\gamma)} -\frac{ 1 +\log c +\frac{1}{2}\log(\log(1/\gamma))   }{\sqrt{\log(1/\gamma)}} \bigg\}\bigg).
\end{align*}
Since the exponent diverges as $\gamma \rightarrow 0^{+}$, we prove the lemma.

\subsection{Proof of Lemma~\ref{lem:sup}} \label{pf:sup}
WLOG, assume that $\|C\| = \sup _{\mathbf{x}\in B} \mathbf{x}^T \mathbf{C} \mathbf{x}$; the case $\|C\| = -\inf _{\mathbf{x}\in B} \mathbf{x}^T \mathbf{C} \mathbf{x}$ follows similarly. Observe that the diameter of each cell resulting from discretization is $\delta$. For a vector $\mathbf{d} $ such that $\|\mathbf{d}\|_2 \leq \delta$ and $\mathbf{x}\in B$,
$(\mathbf{x}+\mathbf{d})^T \mathbf{C} (\mathbf{x}+\mathbf{d}) - \mathbf{x}^T \mathbf{C} \mathbf{x} = 2\mathbf{d}^T\mathbf{C}  \mathbf{x} +\mathbf{d}^T\mathbf{C} \mathbf{d}$.
Thus, we get:
\begin{align*}
&\left|(\mathbf{x}+\mathbf{d})^T \mathbf{C} (\mathbf{x}+\mathbf{d})\right| - \left|\mathbf{x}^T \mathbf{C} \mathbf{x} \right| \\
&\leq
\left|(\mathbf{x}+\mathbf{d})^T \mathbf{C} (\mathbf{x}+\mathbf{d}) - \mathbf{x}^T \mathbf{C} \mathbf{x} \right| \leq  2\left|\mathbf{d}^T \mathbf{C}  \mathbf{x}\right| + \left|\mathbf{d}^T \mathbf{C} \mathbf{d} \right| \\
&\leq  2\|\mathbf{d}\| \|\mathbf{C}\| \|\mathbf{x}\|+ \|\mathbf{C}\| \|\mathbf{d}\|^2
\leq  2\|\mathbf{d}\| \|\mathbf{C}\|  + \|\mathbf{C}\| \|\mathbf{d}\|^2 \\
&\leq 3\|\mathbf{d}\| \|\mathbf{C} \| \leq 3\delta \|\mathbf{C} \| \,.
\end{align*}	
Let $\mathbf{x}^* =\arg \sup _{\mathbf{x}\in B} \mathbf{x}^T\mathbf{C} \mathbf{x}$.
Then, there exists $\mathbf{x}_0\in D_{\delta}$ such that $\|\mathbf{x}_0- \mathbf{x}^*\| \leq \delta$, so
\begin{align*}
&\|\mathbf{C}\| = (\mathbf{x}^*)^T \mathbf{C} \mathbf{x}^* \leq (\mathbf{x}_0)^T \mathbf{C} \mathbf{x}_0 + 3\delta \|\mathbf{C} \| \\
&\leq  \sup_{\mathbf{x}\in D_{\delta}} |\mathbf{x}^T\mathbf{C} \mathbf{x}| + 3\delta \|\mathbf{C}\|\,.
\end{align*}
By rearrangement, we get: $(1-3\delta) \|\mathbf{C} \| \leq \sup _{\mathbf{x}\in D_{\delta}} |\mathbf{x}^T\mathbf{C} \mathbf{x}|$.

\subsection{Proof of Lemma~\ref{nolargedeg}}\label{pf:nolargedeg}

Let us say node $i$ is \emph{bad} if $\den_{\mathbf{A}}(i,[n]) \geq \ctt\cdot n\nu$ for some constant $\ctt$ to be chosen later.
Let $\delta := (n\nu)^{-3}$.
\begin{align*}
&\Pr(\text{there are more than $\delta n$ bad nodes}) \\
&\leq \sum_{X\subset [n]: |X|=\delta n}\Pr\left(\text{every node in $X$ is bad}\right)\\
&\leq \sum_{X\subset [n]: |X|=\delta n}\Pr\left(\den_{\mathbf{A}}(X,[n]) \geq \delta n \cdot  (\ctt\cdot n \nu) \right) \,.
\end{align*}
Note that for any subsets $\mathcal{A}$ and $\mathcal{B}$,
\begin{align}
&\den_{\mathbf{A}}(\mathcal{A},\mathcal{B}) = \sum_{i\in \mathcal{A}} \sum_{j\in \mathcal{B}} \mathbf{A}_{i,j} = \sum_{i\in \mathcal{A}} \sum_{j\in \mathcal{B}} \sum_{ \substack{e\in\mathcal{E} \\\{i,j\}\subset e}} W_e \nn\\
&= \sum_{ e\in \mathcal{E} } \bigg[W_e \bigg(\sum_{(i,j)\in \mathcal{A}\times \mathcal{B}: \{i,j\}\subset e}1\bigg) \bigg]\,, \label{whylem6}
\end{align}
i.e., $\den_{\mathbf{A}}(\mathcal{A},\mathcal{B})$ is a sum of independent random variables taking values in $[0, d^2]$.
Hence, using the fact that  $\ex[ \den_{\mathbf{A}}(X,[n])] \leq \delta n^2 \nu$ ($\because~ \nu\geq \max_{i,j}\ex\left[\mathbf{A}_{i,j}\right]$) together with Lemma~\ref{machinery} (take $b=d^2$), there exists $\cseven>0$ such that 
\begin{align*}
&\Pr\left(  \den_{\mathbf{A}}(X,[n]) \geq k \delta n^2 \nu \right) \leq \exp\left(- \left(\frac{1}{2d^2} k\log k\right) \cdot  \delta n^2\nu \right)\\
& ~~\text{whenever $k\geq \cseven$.}
\end{align*}
By taking $\ctt = \cseven$,
\begin{align*}
&\sum_{X\subset [n]: |X|=\delta n}	\Pr\left(  \den_{\mathbf{A}}(X,[n]) \geq \ctt \delta n^2 \nu \right)  \\
		&\leq \binom{n}{\delta n}\exp\left(- \left(\frac{1}{2d^2} \ctt\log \ctt\right) \cdot  \delta n^2\nu \right) \\
& \overset{(a)}{\leq}  \exp\left\{\left(\delta\log \frac{1}{\delta } + \delta - \left(\frac{1}{2d^2} \ctt\log \ctt\right)\cdot  \delta n\nu\right) \cdot n \right\}\,.
\end{align*}	
where ($a$) is due to the fact that $\binom{n}{m} \leq \left(\frac{ne}{m}\right)^m$.
Plugging back in $\delta = (n\nu )^{-3}$, we obtain
\begin{align}
&\delta\log \frac{1}{\delta } + \delta - \left(\frac{1}{2d^2} \ctt\log \ctt\right)\cdot  \delta n\nu \nn\\
&= -3  \log (n\nu) (n\nu)^{-3} + (n\nu)^{-3} -\left(\frac{1}{2d^2} \ctt\log \ctt\right)\cdot  (n\nu)^{-2}\,. \label{last:4}
\end{align}
Since $\eqref{last:4}\cdot (n\nu)^2 \to -\frac{1}{2d^2} \ctt\log \ctt<0 $ as $n\nu \to \infty$, there exists a constant $\ceight$ such that $n\nu \geq \ceight $ implies $\eqref{last:4}<0$. This completes the proof.

\subsection{Proof of Lemma~\ref{lem:bdd}} \label{pf:bdd}
By taking $\cnine = \ctt$, the first part of Definition~\ref{def:bdd} follows easily by the definition of $\Az$. 

We now turn to the second part of Definition~\ref{def:bdd}.
Without loss of generality, we assume that  $\mathcal{A} \cap \mathcal{B} = \emptyset$ and  $|\mathcal{A}|\leq |\mathcal{B}|$.
\begin{enumerate}
	\item The case where $|\mathcal{B}|\geq \frac{n}{e}$:
	
	It follows that $	\nu|\mathcal{A}||\mathcal{B}|\geq \frac{\nu|\mathcal{A}|n}{e}$, and since we verified the first part of Definition~\ref{def:bdd}, we obtain  $\den_{\Az}(\mathcal{A},\mathcal{B})\leq |\mathcal{A}|\cdot \cnine n\nu$. Hence, $\den_{\Az}(A,B) \leq \cnine e \nu |A||B|$. 
	
	\item The case where $|\mathcal{B}|< \frac{n}{e}$:
	
	It suffices to show the property for the case where $\Az$ is replaced by $\mathbf{A}$ due to the fact that $\den_{\Az}(\mathcal{A},\mathcal{B}) \leq \den_{\mathbf{A}}(\mathcal{A},\mathcal{B})$.
Because of \eqref{whylem6}, $\den_{\mathbf{A}}(\mathcal{A},\mathcal{B})$ is a sum of independent random variables taking values in $[0,d^2]$. 
	As $\ex [\den_{\mathbf{A}}(\mathcal{A},\mathcal{B})] \leq \nu|\mathcal{A}||\mathcal{B}|$ ($\because~ \nu\geq \max_{i,j}\ex\left[\mathbf{A}_{i,j}\right]$), Lemma~\ref{machinery} ensures that there exist constants $\cseven>0$ such that 
	\[
	\Pr\big( \den_{\mathbf{A}}(A,B) >k\cdot \nu |A||B|  \big)  \leq \exp\left(-\frac{1}{2d^2}k\log k \cdot\nu |A||B|\right)
	\]
	for any $k\geq \cseven$ regardless of choices of $\mathcal{A}$ and $\mathcal{B}$.
	
	\begin{claim}
		Let 
		\begin{align*}
		k_{a,b} := \max \left\{ \min\left\{k\geq 1~:~k\log k  \geq \frac{14d^2}{\nu a} \log \frac{n}{b}\right\},~ \cseven \right\}\,.
		\end{align*}
		Then, with probability $1-O(n^{-1})$, the following holds: 
		For any two subsets $\mathcal{A}$ and $\mathcal{B}$,
		\begin{align*}
		\den_{\mathbf{A}}(\mathcal{A},\mathcal{B}) \leq k_{|\mathcal{A}|,|\mathcal{B}|} \nu|\mathcal{A}||\mathcal{B}|\,.
		\end{align*}
		
	\end{claim}
	\begin{IEEEproof}
		It is sufficient to prove the following:
$\Pr\left(\bigcup_{\mathcal{A},\mathcal{B}} \left[\den_{\mathbf{A}}(\mathcal{A},\mathcal{B})> k_{|\mathcal{A}|,|\mathcal{B}|} \nu|\mathcal{A}||\mathcal{B}|\right] \right) = O(n^{-1})$. 
Note that in the case of $|\mathcal{A}|=a$ and $|\mathcal{B}|=b$, $\Pr(\den_{\mathbf{A}}(|\mathcal{A}|,|\mathcal{B}|) > k_{\mathcal{A},\mathcal{B}} \cdot \nu |\mathcal{A}||\mathcal{B}|)$ is upper bounded by $\exp(-\frac{1}{2d^2}k_{a,b}\log k_{a,b} \cdot\nu ab)$ as $k_{a,b} \geq \cseven$.
Hence, the union bound yields:
		\begin{align*}
		&\Pr\left(\bigcup_{\mathcal{A},\mathcal{B}} \left[\den_{\mathbf{A}}(\mathcal{A},\mathcal{B})> k_{\mathcal{A},\mathcal{B}}\cdot \nu|\mathcal{A}||\mathcal{B}|\right] \right) \\
		&{\leq}  \sum_{a,b} \left[\binom{n}{a} \binom{n}{b}\exp\left(-\frac{1}{2d^2}k_{a,b}\log k_{a,b} \cdot\nu ab\right)\right]\,.
		\end{align*} 
		Since there are at most $n^2$ choices for $(a,b)$, it is enough to show that $\binom{n}{a} \binom{n}{b}\exp\left(-\frac{1}{2d^2}k_{a,b}\log k_{a,b} \cdot\nu ab\right) \leq \frac{1}{n^3}$ for any $(a, b)$.
		
		By the definition of ``$k_{a,b}$'', we have $k_{a,b} \log k_{a,b} \geq \frac{14d^2}{\nu a} \log \frac{n}{b}$.
		Hence, 
		\begin{align*}
		&\frac{1}{2d^2}	k_{a,b} \log k_{a,b} \cdot \nu ab \geq  7b \log \frac{n}{b}\overset{(a)}{\geq } a + b +5  b \log \frac{n}{b}\\
		&\overset{(b)}{\geq} a + b + a\log {n}{a} +b\log {n}{b} +3 \log n\,,
		\end{align*}
		where ($a$) follows since $a\leq b\leq \frac{n}{e}$; ($b$) follows since $x\log x$ is increasing on $[1, \frac{n}{e}]$.
		Thus, we have
		\begin{align*}
		&\exp\left(-\frac{1}{2d^2}k_{a,b}\log k_{a,b} \cdot\nu ab\right)\\
		&\leq \exp\left(- a\left(\log \frac{n}{a} +1\right) -b\left(\log \frac{n}{b}+1\right) -3 \log n\right).
		\end{align*}
		Further, since $\binom{n}{m}\leq \left(\frac{ne}{m}\right)^m= \exp{(m(\log{n/m}+1))}$, 
		\begin{align*}
\binom{n}{a} \binom{n}{b} = \exp\left(a\left(\log{n/a} +1\right) + b\left(\log{n/b}+1\right)\right).
		\end{align*}
Thus, $\binom{n}{a} \binom{n}{b}e^{-\frac{1}{2d^2}k_{a,b}\log k_{a,b} \cdot\nu ab} \leq e^{-3\log n}$.
	\end{IEEEproof}
	
	By the above claim, we have that for any $\mathcal{A}$ and $\mathcal{B}$ such that $|\mathcal{A}|\leq |\mathcal{B}|\leq \frac{n}{e}$, either of the following holds: 
	\begin{align*}
	&\text{(i)}~ \den_{\mathbf{A}}(\mathcal{A},\mathcal{B}) \leq \cseven \nu|\mathcal{A}||\mathcal{B}|~~\text{or;}\\
	&\text{(ii)}~k_{|\mathcal{A}|,|\mathcal{B}|} \log k_{|\mathcal{A}|,|\mathcal{B}|}   =  \frac{14d^2}{\nu a} \log \frac{n}{b}\,.
	\end{align*}
	For (ii), one can derive:
$\den_{\mathbf{A}}(\mathcal{A},\mathcal{B})
\leq k_{|\mathcal{A}|,|\mathcal{B}|} \cdot \nu |\mathcal{A}| |\mathcal{B}|= \left( \frac{14d^2}{ \nu|\mathcal{A}| \log k_{|\mathcal{A}|,|\mathcal{B}|}} \log \frac{n}{|\mathcal{B}|}\right)\cdot \nu |\mathcal{A}| |\mathcal{B}|
\leq  \left(\frac{14d^2}{\nu|\mathcal{A}| \log \frac{\den_{\mathbf{A}}(\mathcal{A},\mathcal{B})}{\nu|\mathcal{A}||\mathcal{B}|}}\log \frac{n}{|\mathcal{B}|} \right)\cdot \nu |\mathcal{A}| |\mathcal{B}|$,
	and hence $\den_{\mathbf{A}}(\mathcal{A},\mathcal{B}) \log \frac{\den_{\mathbf{A}}(\mathcal{A},\mathcal{B})}{\nu|\mathcal{A}||\mathcal{B}|} \leq 14d^2 |\mathcal{B}| \log \frac{n}{|\mathcal{B}|}$.
\end{enumerate}
Combining the above two cases 1) and 2), the proof is completed by taking $\cten = \max\{\cnine e,~ \cseven  \}$ and $\celeven = 14d^2$.

\subsection{Proof of Lemma~\ref{lem:comp}} \label{pf:comp}
One can easily show that the LHS is equal to $\sum_{e\in \mathcal{E}: W_e\neq \era}\left[\ind \left\{ \ye(\mathbf{X}\oplus\mathbf{e}_1) \neq W_e \right\} -\ind \left\{ \ye(\mathbf{X}) \neq W_e \right\} \right]$.
Since the summand is nonzero only if $\ye(\mathbf{X}\oplus\mathbf{e}_1)\neq \ye(\mathbf{X})$, we count the number of such edges.

First, observe that if $1 \notin e$, $\ye(\mathbf{X}\oplus\mathbf{e}_1) = \ye(\mathbf{X})$.
Further, if two (or more) nodes other than node $1$ are of different affiliations, then $\ye(\mathbf{X}\oplus\mathbf{e}_1) = \ye(\mathbf{X}) = 0$.
Thus, $e$ must include $1$ and all the other nodes in $e$ must be of the same affiliation: 
If all the nodes of $e$ other than node $1$ are affiliated with community $0$, $\ye(\mathbf{X}\oplus\mathbf{e}_1)=1$ and $\ye(\mathbf{X})=0$; and 
if all the nodes of $e$ other than node $1$ are affiliated with community $1$, $\ye(\mathbf{X}\oplus\mathbf{e}_1)=0$ and $\ye(\mathbf{X})=1$.

Define the set of edges corresponding to the former case as $\mathcal{E}_1$, and that corresponding to the latter case as $\mathcal{E}_2$, i.e., $\mathcal{E}_1:= \{e \in \mathcal{E}~:~1\in e~~\text{and}~~(e\setminus \{1\})\subset \{\ee n+1 , \ee n+2 , \ldots, \ee n + n/2 \} \}$ and $\mathcal{E}_2:= \{e \in \mathcal{E}~:~1\in e~~\text{and}~~(e\setminus \{1\})\subset \{2,3,\ldots, \ee n, \ee n +n/2+1, \ee n +n/2+2, \ldots, n \} \}$.
Consider all homogeneous edges in $\mathcal{E}_1$. 
The total contribution of the terms associated with these edges to the sum is $\sum_{e\in \mathcal{E}_1~: W_e\neq \era, e:\text{homogeneous}}\left[\ind \left\{ 1 \neq W_e \right\} -\ind \left\{ 0 \neq W_e \right\} \right]$.
Each term is $-1$ if observation is not corrupted, and $+1$ if observation is corrupted.
Thus, the total contribution is $\sum_{i=1}^{\left| \{ e\in \mathcal{E}_1 ~:~ e \text{ is homogeneous} \}\right|} P_i(2\Theta_i-1) = \sum_{i=1}^{\binom{n/2-\ee n}{d-1}} P_i(2\Theta_i-1)$, where $P_i \overset{\text{i.i.d.}}{\sim} \bern(\alpha_n) $ and $\Theta_i \overset{\text{i.i.d.}}{\sim} \bern(\theta)$.
By rewriting other contributions in a similar way, we complete the proof.

\subsection{Proof of Lemma~\ref{lem:3}}
\label{pf:3}
Let $Z:= \cmatrix \sum_{i=1}^K P_i(2\Theta_i-1) +\ell$ and $\mgf(\lambda) := \ex[e^{\lambda \cmatrix P_1(2\Theta_1-1)}]$. 
Via simple calculation, we have
\begin{align*}
&\Pr\left( Z>0 \right) = \Pr\big(e^{\frac{1}{2} Z} > 1\big) \leq  e^{\frac{1}{2}\ell}\big\{\mgf\big( 1/2\big) \big\}^K \\
&= e^{\frac{1}{2}\ell + K \left\{-\alpha_n(\sqrt{1-\theta}-\sqrt{\theta})^2 + O(\alpha_n^2)\right\}.}.
\end{align*}
\vspace{-0.2in}
\small
\bibliographystyle{IEEEtran}
\bibliography{ref}

\begin{thebibliography}{10}
\providecommand{\url}[1]{#1}
\csname url@samestyle\endcsname
\providecommand{\newblock}{\relax}
\providecommand{\bibinfo}[2]{#2}
\providecommand{\BIBentrySTDinterwordspacing}{\spaceskip=0pt\relax}
\providecommand{\BIBentryALTinterwordstretchfactor}{4}
\providecommand{\BIBentryALTinterwordspacing}{\spaceskip=\fontdimen2\font plus
\BIBentryALTinterwordstretchfactor\fontdimen3\font minus
  \fontdimen4\font\relax}
\providecommand{\BIBforeignlanguage}[2]{{%
\expandafter\ifx\csname l@#1\endcsname\relax
\typeout{** WARNING: IEEEtran.bst: No hyphenation pattern has been}%
\typeout{** loaded for the language `#1'. Using the pattern for}%
\typeout{** the default language instead.}%
\else
\language=\csname l@#1\endcsname
\fi
#2}}
\providecommand{\BIBdecl}{\relax}
\BIBdecl

\bibitem{Ahn1706:Information}
K.~Ahn, K.~Lee, and C.~Suh, ``Information-theoretic limits of subspace
  clustering,'' in \emph{IEEE ISIT}, 2017.

\bibitem{shi2000normalized}
J.~Shi and J.~Malik, ``Normalized cuts and image segmentation,'' \emph{IEEE
  TPAMI}, vol.~22, no.~8, pp. 888--905, 2000.

\bibitem{mcsherry2001spectral}
F.~McSherry, ``Spectral partitioning of random graphs,'' in \emph{FOCS}.\hskip
  1em plus 0.5em minus 0.4em\relax IEEE, 2001, pp. 529--537.

\bibitem{rohe2011spectral}
K.~Rohe, S.~Chatterjee, and B.~Yu, ``Spectral clustering and the
  high-dimensional stochastic blockmodel,'' \emph{The Annals of Statistics},
  2011.

\bibitem{lei2015consistency}
J.~Lei and A.~Rinaldo, ``Consistency of spectral clustering in stochastic block
  models,'' \emph{The Annals of Statistics}, vol.~43, no.~1, pp. 215--237,
  2015.

\bibitem{ghoshal2009random}
G.~Ghoshal, V.~Zlati{\'c}, G.~Caldarelli, and M.~Newman, ``Random hypergraphs
  and their applications,'' \emph{Physical Review E}, vol.~79, no.~6, p.
  066118, 2009.

\bibitem{michoel2012alignment}
T.~Michoel and B.~Nachtergaele, ``Alignment and integration of complex networks
  by hypergraph-based spectral clustering,'' \emph{Physical Review E}, vol.~86,
  no.~5, p. 056111, 2012.

\bibitem{JMLR:v18:16-100}
D.~Ghoshdastidar and A.~Dukkipati, ``Uniform hypergraph partitioning: Provable
  tensor methods and sampling techniques,'' \emph{JMLR}, vol.~18, no.~50, pp.
  1--41, 2017.

\bibitem{ghoshdastidar2014consistency}
------, ``Consistency of spectral partitioning of uniform hypergraphs under
  planted partition model,'' in \emph{NIPS}, 2014, pp. 397--405.

\bibitem{ghoshdastidar2015provable}
------, ``A provable generalized tensor spectral method for uniform hypergraph
  partitioning.'' in \emph{ICML}, 2015, pp. 400--409.

\bibitem{ghoshdastidar2015consistency}
------, ``Consistency of spectral hypergraph partitioning under planted
  partition model,'' \emph{The Annals of Statistics, 45(1), pp. 289-315}, 2017.

\bibitem{florescu2015spectral}
L.~Florescu and W.~Perkins, ``Spectral thresholds in the bipartite stochastic
  block model,'' \emph{COLT}, pp. 943--959, 2016.

\bibitem{abbe2017community}
E.~Abbe, ``Community detection and stochastic block models: Recent
  developments,'' \emph{JMLR, Special Issue}, 2017.

\bibitem{decelle}
A.~Decelle, F.~Krzakala, C.~Moore, and L.~Zdeborov{\'a}, ``Asymptotic analysis
  of the stochastic block model for modular networks and its algorithmic
  applications,'' \emph{Physical Review E}, 2011.

\bibitem{mossel2015reconstruction}
E.~Mossel, J.~Neeman, and A.~Sly, ``Reconstruction and estimation in the
  planted partition model,'' \emph{Probability Theory and Related Fields}, vol.
  162, no. 3-4, pp. 431--461, 2015.

\bibitem{Mas14}
L.~Massouli{\'e}, ``Community detection thresholds and the weak ramanujan
  property,'' in \emph{STOC}.\hskip 1em plus 0.5em minus 0.4em\relax ACM, 2014,
  pp. 694--703.

\bibitem{bordenave}
C.~Bordenave, M.~Lelarge, and L.~Massoulie, ``Non-backtracking spectrum of
  random graphs: Community detection and non-regular ramanujan graphs,'' in
  \emph{FOCS}.\hskip 1em plus 0.5em minus 0.4em\relax IEEE, 2015, pp.
  1347--1357.

\bibitem{YC14}
Y.~Chen and J.~Xu, ``Statistical-computational tradeoffs in planted problems
  and submatrix localization with a growing number of clusters and
  submatrices,'' \emph{JMLR}, vol.~17, no.~1, pp. 882--938, 2016.

\bibitem{NN14}
J.~Neeman and P.~Netrapalli, ``Non-reconstructability in the stochastic block
  model,'' \emph{arXiv preprint arXiv:1404.6304}, 2014.

\bibitem{Mon15}
A.~Montanari, ``Finding one community in a sparse graph,'' \emph{Journal of
  Statistical Physics}, vol. 161, no.~2, pp. 273--299, 2015.

\bibitem{BM16}
J.~Banks, C.~Moore, J.~Neeman, and P.~Netrapalli, ``Information-theoretic
  thresholds for community detection in sparse networks,'' in \emph{COLT},
  2016.

\bibitem{abbe2015detection}
E.~Abbe and C.~Sandon, ``Proof of the achievability conjectures in the general
  stochastic block model,'' \emph{CPAM}, 2017.

\bibitem{AL14}
A.~A. Amini and E.~Levina, ``On semidefinite relaxations for the block model,''
  \emph{arXiv preprint arXiv:1406.5647}, 2014.

\bibitem{gao2015achieving}
C.~Gao, Z.~Ma, A.~Y. Zhang, and H.~H. Zhou, ``Achieving optimal
  misclassification proportion in stochastic block model,'' \emph{JMLR}, 2017.

\bibitem{MNS14a}
E.~Mossel, J.~Neeman, and A.~Sly, ``Consistency thresholds for binary symmetric
  block models,'' \emph{arXiv preprint arXiv:1407.1591}, 2014.

\bibitem{yun2014accurate}
S.-Y. Yun and A.~Proutiere, ``Accurate community detection in the stochastic
  block model via spectral algorithms,'' \emph{arXiv}, 2014.

\bibitem{abbe2015community}
E.~Abbe and C.~Sandon, ``Community detection in general stochastic block
  models: Fundamental limits and efficient algorithms for recovery,'' in
  \emph{FOCS}.\hskip 1em plus 0.5em minus 0.4em\relax IEEE, 2015, pp. 670--688.

\bibitem{abbe2016exact}
E.~Abbe, A.~S. Bandeira, and G.~Hall, ``Exact recovery in the stochastic block
  model,'' \emph{IEEE Transactions on Information Theory}, 2016.

\bibitem{7523889}
B.~Hajek, Y.~Wu, and J.~Xu, ``Achieving exact cluster recovery threshold via
  semidefinite programming: Extensions,'' \emph{IEEE Transactions on
  Information Theory}, vol.~62, no.~10, pp. 5918--5937, Oct 2016.

\bibitem{zhang2016minimax}
A.~Y. Zhang, H.~H. Zhou \emph{et~al.}, ``Minimax rates of community detection
  in stochastic block models,'' \emph{The Annals of Statistics}, 2016.

\bibitem{angelini}
M.~Angelini, F.~Caltagirone, F.~Krzakala, and L.~Zdeborova, ``Spectral
  detection on sparse hypergraphs,'' in \emph{Allerton}, 2015.

\bibitem{ahn2016community}
K.~Ahn, K.~Lee, and C.~Suh, ``Community recovery in hypergraphs,''
  \emph{ArXiv}, 2016.

\bibitem{wangetal}
C.-Y. Lin, C.~I, and I.-H. Wang, ``On the fundamental statistical limit of
  community detection in random hypergraphs,'' in \emph{IEEE ISIT}, 2017.

\bibitem{govindu2005tensor}
V.~M. Govindu, ``A tensor decomposition for geometric grouping and
  segmentation,'' in \emph{IEEE CVPR}, 2005.

\bibitem{agarwal2006higher}
S.~Agarwal, K.~Branson, and S.~Belongie, ``Higher order learning with graphs,''
  in \emph{ICML}.\hskip 1em plus 0.5em minus 0.4em\relax ACM, 2006, pp. 17--24.

\bibitem{chen2009spectral}
G.~Chen and G.~Lerman, ``Spectral curvature clustering (scc),''
  \emph{International Journal of Computer Vision}, vol.~81, no.~3, pp.
  317--330, 2009.

\bibitem{hitchcock1927expression}
F.~L. Hitchcock, ``The expression of a tensor or a polyadic as a sum of
  products,'' \emph{Studies in Applied Mathematics}, 1927.

\bibitem{barak2016noisy}
B.~Barak and A.~Moitra, ``Noisy tensor completion via the sum-of-squares
  hierarchy,'' in \emph{COLT}, 2016, pp. 417--445.

\bibitem{kim2017community}
C.~Kim, A.~S. Bandeira, and M.~X. Goemans, ``Community detection in
  hypergraphs, spiked tensor models, and sum-of-squares,'' \emph{arXiv}, 2017.

\bibitem{jog2015information}
V.~Jog and P.-L. Loh, ``Information-theoretic bounds for exact recovery in
  weighted stochastic block models using the renyi divergence,'' \emph{arXiv
  preprint arXiv:1509.06418}, 2015.

\bibitem{xu2017optimal}
M.~Xu, V.~Jog, and P.-L. Loh, ``Optimal rates for community estimation in the
  weighted stochastic block model,'' \emph{arXiv preprint arXiv:1706.01175},
  2017.

\bibitem{holland1983stochastic}
P.~W. Holland, K.~B. Laskey, and S.~Leinhardt, ``Stochastic blockmodels: First
  steps,'' \emph{Social networks}, vol.~5, no.~2, pp. 109--137, 1983.

\bibitem{mossel2015consistency}
E.~Mossel, J.~Neeman, and A.~Sly, ``Consistency thresholds for the planted
  bisection model,'' in \emph{STOC}.\hskip 1em plus 0.5em minus 0.4em\relax
  ACM, 2015, pp. 69--75.

\bibitem{feige2005spectral}
U.~Feige and E.~Ofek, ``Spectral techniques applied to sparse random graphs,''
  \emph{Random Structures \& Algorithms}, 2005.

\bibitem{coja2010graph}
A.~Coja-Oghlan, ``Graph partitioning via adaptive spectral techniques,''
  \emph{Combinatorics, Probability and Computing}, 2010.

\bibitem{vu2014simple}
V.~Vu, ``A simple svd algorithm for finding hidden partitions,'' \emph{arXiv
  preprint arXiv:1404.3918}, 2014.

\bibitem{guedon2016community}
O.~Gu{\'e}don and R.~Vershynin, ``Community detection in sparse networks via
  grothendieck’s inequality,'' \emph{Probability Theory and Related Fields},
  vol. 165, no. 3-4, pp. 1025--1049, 2016.

\bibitem{matouvsek2000approximate}
J.~Matou{\v{s}}ek, ``On approximate geometric k-clustering,'' \emph{Discrete \&
  Computational Geometry}, vol.~24, no.~1, pp. 61--84, 2000.

\bibitem{chin2015stochastic}
P.~Chin, A.~Rao, and V.~Vu, ``Stochastic block model and community detection in
  sparse graphs: A spectral algorithm with optimal rate of recovery.'' in
  \emph{COLT}, 2015, pp. 391--423.

\bibitem{boutsidis2015spectral}
C.~Boutsidis, P.~Kambadur, and A.~Gittens, ``Spectral clustering via the power
  method-provably,'' in \emph{ICML}, 2015, pp. 40--48.

\bibitem{keshavan2010matrix}
R.~H. Keshavan, A.~Montanari, and S.~Oh, ``Matrix completion from a few
  entries,'' \emph{IEEE Transactions on Information Theory}, vol.~56, no.~6,
  pp. 2980--2998, 2010.

\bibitem{jain2013low}
P.~Jain, P.~Netrapalli, and S.~Sanghavi, ``Low-rank matrix completion using
  alternating minimization,'' in \emph{STOC}.\hskip 1em plus 0.5em minus
  0.4em\relax ACM, 2013, pp. 665--674.

\bibitem{netrapalli2013phase}
P.~Netrapalli, P.~Jain, and S.~Sanghavi, ``Phase retrieval using alternating
  minimization,'' in \emph{NIPS}, 2013, pp. 2796--2804.

\bibitem{candes2015phase}
E.~J. Candes, X.~Li, and M.~Soltanolkotabi, ``Phase retrieval via wirtinger
  flow: Theory and algorithms,'' \emph{IEEE Transactions on Information
  Theory}, vol.~61, no.~4, pp. 1985--2007, 2015.

\bibitem{yi2016fast}
X.~Yi, D.~Park, Y.~Chen, and C.~Caramanis, ``Fast algorithms for robust pca via
  gradient descent,'' in \emph{NIPS}, 2016, pp. 4152--4160.

\bibitem{chen2016community}
Y.~Chen, G.~Kamath, C.~Suh, and D.~Tse, ``Community recovery in graphs with
  locality,'' in \emph{ICML}, 2016.

\bibitem{balakrishnan2017}
S.~Balakrishnan, M.~J. Wainwright, and B.~Yu, ``Statistical guarantees for the
  em algorithm: From population to sample-based analysis,'' 2017.

\bibitem{chen2015spectral}
Y.~Chen and C.~Suh, ``Spectral {MLE}: Top-$k$ rank aggregation from pairwise
  comparisons,'' in \emph{ICML}, 2015, pp. 371--380.

\bibitem{Krzakala24122013}
F.~Krzakala, C.~Moore, E.~Mossel, J.~Neeman, A.~Sly, L.~Zdeborová, and
  P.~Zhang, ``Spectral redemption in clustering sparse networks,'' \emph{PNAS},
  vol. 110, no.~52, pp. 20\,935--20\,940, 2013.

\bibitem{chien2018minimax}
I.~Chien, C.-Y. Lin, I.~Wang \emph{et~al.}, ``On the minimax misclassification
  ratio of hypergraph community detection,'' \emph{arXiv}, 2018.

\bibitem{coja2007counting}
A.~Coja-Oghlan, C.~Moore, and V.~Sanwalani, ``Counting connected graphs and
  hypergraphs via the probabilistic method,'' \emph{Random Structures \&
  Algorithms}, vol.~31, no.~3, pp. 288--329, 2007.

\bibitem{COOLEY2015569}
J.~Nesetril, O.~Serra, J.~A. Telle, O.~Cooley, M.~Kang, and C.~Koch,
  ``Evolution of high-order connected components in random hypergraphs,''
  \emph{Electronic Notes in Discrete Mathematics}, 2015.

\bibitem{friedman1989second}
J.~Friedman, J.~Kahn, and E.~Szemeredi, ``On the second eigenvalue of random
  regular graphs,'' in \emph{STOC}.\hskip 1em plus 0.5em minus 0.4em\relax ACM,
  1989, pp. 587--598.

\bibitem{tao2012topics}
T.~Tao, \emph{Topics in random matrix theory}.\hskip 1em plus 0.5em minus
  0.4em\relax American Mathematical Society Providence, RI, 2012, vol. 132.

\bibitem{tropp2012user}
J.~A. Tropp, ``User-friendly tail bounds for sums of random matrices,''
  \emph{Foundations of computational mathematics}, 2012.

\bibitem{jain2014provable}
P.~Jain and S.~Oh, ``Provable tensor factorization with missing data,'' in
  \emph{NIPS}, 2014, pp. 1431--1439.

\bibitem{alon2004probabilistic}
N.~Alon and J.~H. Spencer, \emph{The probabilistic method}, 2004.

\bibitem{alon1998finding}
N.~Alon, M.~Krivelevich, and B.~Sudakov, ``Finding a large hidden clique in a
  random graph,'' \emph{Random Structures and Algorithms}, 1998.

\bibitem{tron2007benchmark}
R.~Tron and R.~Vidal, ``A benchmark for the comparison of 3-d motion
  segmentation algorithms,'' in \emph{IEEE CVPR}, 2007.

\bibitem{heckel2017dimensionality}
R.~Heckel, M.~Tschannen, and H.~B{\"o}lcskei, ``Dimensionality-reduced subspace
  clustering,'' \emph{Information and Inference: A Journal of the IMA}, vol.~6,
  no.~3, pp. 246--283, 2017.

\bibitem{elhamifar2013sparse}
E.~Elhamifar and R.~Vidal, ``Sparse subspace clustering: Algorithm, theory, and
  applications,'' \emph{IEEE TPAMI}, vol.~35, no.~11, pp. 2765--2781, 2013.

\bibitem{dyer2013greedy}
E.~L. Dyer, A.~C. Sankaranarayanan, and R.~G. Baraniuk, ``Greedy feature
  selection for subspace clustering.'' \emph{JMLR}, vol.~14, no.~1, pp.
  2487--2517, 2013.

\bibitem{liu2013robust}
G.~Liu, Z.~Lin, S.~Yan, J.~Sun, Y.~Yu, and Y.~Ma, ``Robust recovery of subspace
  structures by low-rank representation,'' \emph{IEEE TPAMI}, 2013.

\bibitem{heckel2015robust}
R.~Heckel and H.~B{\"o}lcskei, ``Robust subspace clustering via thresholding,''
  \emph{IEEE Transactions on Information Theory}, vol.~61, no.~11, pp.
  6320--6342, 2015.

\bibitem{park2014greedy}
D.~Park, C.~Caramanis, and S.~Sanghavi, ``Greedy subspace clustering,'' in
  \emph{NIPS}, 2014, pp. 2753--2761.

\bibitem{jain2013efficient}
S.~Jain and V.~Madhav~Govindu, ``Efficient higher-order clustering on the
  grassmann manifold,'' in \emph{IEEE ICCV}, 2013, pp. 3511--3518.

\bibitem{duchenne2011tensor}
O.~Duchenne, F.~Bach, I.-S. Kweon, and J.~Ponce, ``A tensor-based algorithm for
  high-order graph matching,'' \emph{IEEE TPAMI}, vol.~33, no.~12, pp.
  2383--2395, 2011.

\bibitem{yang2015defining}
J.~Yang and J.~Leskovec, ``Defining and evaluating network communities based on
  ground-truth,'' \emph{Knowledge and Information Systems}, vol.~42, no.~1, pp.
  181--213, 2015.

\end{thebibliography}
\begin{IEEEbiographynophoto}{Kwangjun Ahn}
	 received his B.S. degree in the Department of Mathematical Sciences from Korea Advanced Institute of Science and Technology (KAIST) in 2017. 
	 He is currently a military police desk clerk in the US Army as a part of Korean Augmentation to the US Army (KATUSA). 
	His research interests lie in applied mathematics. 
\end{IEEEbiographynophoto}
\begin{IEEEbiographynophoto}{Kangwook Lee} is a postdoctoral researcher in the School of Electrical Engineering at Korea Advanced Institute of Science and Technology (KAIST). He earned his Ph.D. in EECS from UC Berkeley in 2016. He is a recipient of the KFAS Fellowship from 2010 to 2015. His research interests lie in information theory and machine learning.
\end{IEEEbiographynophoto}
\begin{IEEEbiographynophoto}{Changho Suh}(S'10--M'12) is an Ewon Associate Professor in the School of Electrical Engineering at Korea Advanced Institute of Science and Technology (KAIST) since 2012. He received the B.S. and M.S. degrees in Electrical Engineering from KAIST in 2000 and 2002 respectively, and the Ph.D. degree in Electrical Engineering and Computer Sciences from UC-Berkeley in 2011. From 2011 to 2012, he was a postdoctoral associate at the Research Laboratory of Electronics in MIT. From 2002 to 2006, he had been with the Telecommunication R\&D Center, Samsung Electronics. Dr. Suh received the 2015 Haedong Young Engineer Award from the Institute of Electronics and Information Engineers, the 2013 Stephen O. Rice Prize from the IEEE Communications Society, the David J. Sakrison Memorial Prize from the UC-Berkeley EECS Department in 2011, and the Best Student Paper Award of the IEEE International Symposium on Information Theory in 2009. 
\end{IEEEbiographynophoto}
\end{document}